\documentclass{amsart}
\usepackage{commands}

\newcommand{\cmt}[1]{}

\begin{document}
\title{Local newforms and spherical characters for unitary groups}
\author{Gefei Dang}

\begin{abstract} 
We prove a smooth transfer statement analogous to Jacquet--Rallis's fundamental lemma, use it to compute a local spherical character appearing in the Ichino--Ikeda conjecture, and prove a statement on the existence of local newforms for unitary groups as a corollary.  
\end{abstract}

\maketitle 

\tableofcontents

\section{Introduction}
\label{section_intro}

Let $p\neq 2$ be an odd prime. Let $\index{F@$F$}F$ be a finite extension of $\Q_p$ and \index{E@$E$}$E/F$ be an unramified quadratic extension of fields. Denote by \index{qe@$q_E$} \index{qf@$q_F$} $q_E$ and $q_F$ the sizes of the residue fields of $\roi_E$ and $\roi_F$ respectively. Assume $q_F>n$ for technical reasons. Our first main result is the following.

\begin{thm}
\label{exist_newform}
Let $V$ be a hermitian space over $E$. In every tempered Vogan $L$-packet of $U(V)(F)$ there exists a unique representation that contains newforms.
\end{thm}

When $V$ is odd-dimensional, \cite{Atobe-Oi-Yasuda} showed that this unique representation containing newforms is the generic one. When $\dim V=2m$ is even, our proof shows that this representation is determined by local root numbers and is the $\pi$ such that $\Hom_{U(1)\ltimes N}(\pi,\C)\neq 0$ appearing in the (known) Gan--Gross--Prasad conjecture for $U(1) \times U(2m)$, where $U(1) \ltimes N$ is the Bessel subgroup defined in \cite[Section 12]{GGP} (cf. loc. cit. Section 17).

There is no uniform definition of newforms for unitary groups and the precise definition of our newforms is given right before Theorem \ref{thm:mainthm}. Newforms are the analog of spherical elements in unramified representations for ramified representations, and various work has been done on newforms for different groups. For example, Jacquet, Piatetski-Shapiro, and Shalika \cite{J-PS-S} established the theory of newforms for $GL_n$, which was subsequently generalized by Atobe, Kondo, and Yasuda \cite{atobe_kondo_yasuda_2022}. Roberts and Schmidt \cite{Roberts-Schmidt} have studied newforms for $\text{GSp}_4$ and Tsai \cite{Tsai} has studied those for split special odd orthogonal groups.
Lansky and Raghuram \cite{Lansky-Raghuram} established the newforms theory for $U(1,1)$.
Atobe, Oi, and Yasuda \cite{Atobe-Oi-Yasuda} recently showed the existence and uniqueness (up to scalar) of local newforms for tempered generic representations of odd unitary groups. Later, Atobe \cite{atobe2023} studied newforms (which are defined differently from us) for quasi-split even unitary groups. 

Our main new results are the existence of newforms for even unitary groups and some explicit information on the newforms when they are unique. 
For example, our argument implies that when $V$ is even-dimensional and the conductor of the Vogan $L$-packet is odd, the packet contains a representation of the non quasi-split unitary group that has newforms.
Our definitions and results are consistent with Gross' conjectures in his unpublished manuscript on newforms for odd unitary groups \cite{Gross}.
Note that for quasi-split even unitary groups, Atobe's newforms are not unique, even in the case of $U(1,1)$.
On the other hand, our definition of newforms guarantees its uniqueness in the 2-dimensional case (see \Cref{section_uniqueness}).
This raises the question of whether our compact group $K_{n+1}$ (defined below) in the definition of newforms is ``better'' than those used by \cite{atobe2023} and \cite{Lansky-Raghuram}. We plan to investigate the uniqueness of newforms for general even $n$ in a future paper.

Using a different approach from \cite{Atobe-Oi-Yasuda} and \cite{atobe2023}, we prove \Cref{exist_newform} by calculating the valuation of a local character $J_\pi$ at a special test vector. 
Let \index{pi@$\vp$}$\vp$ be a uniformizer of $F$. Denote by $\sigma$ the nontrivial element in $\Gal(E/F)$ and denote $\sigma(a)$ by $\bar{a}$ for elements $a\in E$. 
Let \index{V@$V$} $V$ be an $n+1$ dimensional hermitian space over $E$ and let \index{W@$W$} $W\subset V$ be an $n$-dimensional subspace of $V$ with discriminant 1 such that \index{e@$e$} $V=W\oplus \pa{e}$ for some $e$ orthogonal to $W$ satisfying \index{epsilon@$\epsilon$} $\pa{e,e}=\vp^\epsilon$ with $\epsilon \in \{0,1\}$. Recall that up to isomorphism there are only two distinct $n+1$ dimensional hermitian spaces on $E/F$, distinguished by $\epsilon$. From now on, all representations will be smooth admissible.

Let $U(V)$ and $U(W)$ be the corresponding unitary groups. 
Consider the product \index{G@$G$} \index{H@$H$} $G=U(W) \times U(V)$ and view $H=U(W)$ as a subgroup of $G$ via the diagonal embedding.
We say a representation $\pi$ of $G(F)$ is $H$-distinguished if $\Hom_{H(F)}(\pi,\C)\neq 0$.
Let \index{pi@$\pi$} $\pi=\pi_n \boxtimes \pi_{n+1}$ be an irreducible tempered unitary $H$-distinguished representation of $G(F)$ such that $\pi_n$ is an unramified representation of 
$U(W)(F)$ and $\pi_{n+1}$ is a representation of $U(V)(F)$ with conductor \index{c@$c$} $c>0$ (see \Cref{base_change}).

Following \cite{Zhang14}, we define the local character \index{Jpi@$J_\pi$} $J_\pi: C^\infty_c(G(F)) \to \C$ on the space of Schwartz functions on $G(F)$ by
\[J_\pi(f)=\int_{H(F)} \tr(\pi(h)\pi(f))dh.\]
It comes from a global period integral and becomes the local component in a (nonstandard) formulation of the Ichino--Ikeda conjecture after normalization (see Conjecture 1.6 \textit{loc. cit.}). 
Equivalently, 
\[J_\pi(f)=\sum_{v \in ON(\pi)} \int_{H(F)} \pa{\pi(h)\pi(f)v,v}dh,\]
where $\pa{,}$ is a paring on $\pi$ that makes it unitary and $ON(\pi)$ is an orthonormal basis of $\pi$.
When $\pi$ is unramified and $f=\id_{G(\roi_F)}$ is the characteristic function of $G(\roi_F)$, we have \cite[Section 4.4]{Zhang14}
\begin{equation}
\label{eq_J_unram}
J_\pi(f)=\text{constant}\cdot \frac{L(\frac{1}{2},BC(\pi))}{L(1,\pi,Ad)}
\end{equation}
for some constant independent of $\pi$, where $BC(\pi)$ is the base change of $\pi$ (see \Cref{base_change}) and $L(s,\pi,Ad)$ is the adjoint $L$-function.

We fix an arbitrary self-dual lattice \index{Lambdan@$\Lambda_n$} $\Lambda_n$ in $W$ and let \index{Lambdan1@$\Lambda_{n+1}$} $\Lambda_{n+1} =\Lambda_n+ \vp^{\fl{c/2}}e \roi_E $. 
Define \index{Kn@$K_n$} $K_n=\Stab_{U(W)(F)}(\Lambda_n)$ to be the stabilizer of $\Lambda_n$ in $U(W)(F)$. 
Let $\Lambda_{n+1}^\vee$ be the dual lattice of $\Lambda_{n+1}$, then there is a natural map  $\Stab_{U(V)(F)}(\Lambda_{n+1}) \to \GL(\Lambda^{\vee}_{n+1}/\Lambda_{n+1})$, whose kernel we denote by \index{Kn1@$K^c_{n+1}$} $K^c_{n+1}$. 
Let \index{K@$K^c$} $K^c=K_n \times K^c_{n+1}$.
Define $\sigma_n=BC(\pi_n), \sigma_{n+1}=BC(\pi_{n+1})$ and let $\sigma_u$ be the``unramified part'' of $\sigma_{n+1}$ as defined in \Cref{pi_unram}. Elements in $\pi_{n+1}^{K^c_{n+1}}$ are called newforms in $\pi_{n+1}$.
In this paper, we show the following generalization of \Cref{eq_J_unram}:
\begin{thm}\label{thm:mainthm}
With notations as above, when $\epsilon$ and $c$ are of the same parity,
\[J_\pi(\id_{K^c})=C\frac{L(\frac{1}{2},\sigma_n \times \sigma_{n+1})}{L(1,\sigma_n, As^{(-1)^{n}})L(1,\sigma_u, As^{(-1)^{n+1}})},\]
where
\[C=\vol(K_n)^2L(1,\eta_{E/F})q_F^{-c(n+1)}(1+q_F^{-n}).\]
\end{thm}

\begin{rmk}
The volume of $K_n$ with respect to our Haar measure is given in \Cref{vol_UW_OF}.
\end{rmk}

\begin{rmk}
Note that $L(1,\sigma_n, As^{(-1)^{n}})=L(1,\pi_n,Ad)$. When $\pi_{n+1}$ is unramified, we have $\sigma_u=\sigma_{n+1}$ and $L(1,\sigma_u, As^{(-1)^{n}})=L(1,\pi_{n+1},Ad)$, recovering \Cref{eq_J_unram}. However, in general $\sigma_u$ is not necessarily base change of a representation of $U(V)(F)$.
\end{rmk}

Analogous to \cite{J-PS-S}, this theorem partially determines newforms. When the space of newforms is one dimensional, as $\pi_n$ varies through unramified representations of $U(W)(F)$, the values of $J_\pi(\id_{K^c})$ should determine the matrix coefficients of the newforms when restricted to $K^c_{n+1}U(W)(F)K^c_{n+1}$. So in some sense, our result complements those in \cite{Atobe-Oi-Yasuda,atobe2023}.

Evaluation of period integrals at special test vectors is interesting in its own right, in that it is associated with the integral representations of $L$-functions, and height pairings of special algebraic cycles, the arithmetic analog of period integrals, are used to formulate equalities like the Gross--Zagier formula \cite{Gross-Zagier}, the arithmetic analog of the Waldspurger formula \cite{Waldspurger1985}.
The nonvanishing of $J_\pi(\id_{K^c})$ also becomes a useful local condition in the proof of the $p$-adic Beilinson--Bloch--Kato conjecture and $p$-adic arithmetic Gan--Gross--Prasad conjecture by Disegni and Zhang \cite{DZ}. 

A direct consequence of \Cref{thm:mainthm} is the following.
Recall that the local integral that appears in the Ichino--Ikeda conjecture for unitary groups \cite{Harris} is
\[ \alpha(v, v'):= \int_{H(F)} \pa{\pi(h)v, v'} dh, \quad v, v'\in \pi,\]
where the pairing is the one used in the definition of $J_\pi$.
\begin{cor}
\label{cor_I-I}
With notations as above, if the space of newforms in $\pi_{n+1}$ is one dimensional, then
\[\alpha(\phi, \phi)=\pa{\phi, \phi}J(\id_{K^c})=C\pa{\phi, \phi}\frac{L(\frac{1}{2},\sigma_n \times \sigma_{n+1})}{L(1,\sigma_n, As^{(-1)^{n}})L(1,\sigma_u, As^{(-1)^{n+1}})}\]
for $\phi\in \pi^{K^c}$, where $C$ is as in \Cref{thm:mainthm}. 
\end{cor}

\begin{rmk}
When $n+1$ is odd, the condition dim $\pi_{n+1}^{K^c_{n+1}}=1$ is satisfied exactly when $\pi_{n+1}$ is generic \cite[Theorem 1.1]{Atobe-Oi-Yasuda}. When $n+1$ is even, it is not known under what condition $\pi_{n+1}^{K^c_{n+1}}=1$ holds.
\end{rmk}

The structure of this paper is as follows. We describe in \Cref{notation_n_measure} the notations and measures that will be used throughout this paper.
\Cref{section_fl} focuses on the proof of our transfer statements -- we first recall the definitions and properties of orbital integrals and smooth transfers, and prove a Lie algebra transfer statement (\Cref{lie_alg_fl}) by converting it into the known case of Jacquet--Rallis's fundamental lemma (\Cref{JR_fl}). From there, we derive the inhomogeneous group  (\Cref{inhom_fl}) and the homogeneous group  (\Cref{hom_fl}) statements. 
In \Cref{section_calc_I} we calculate $J_\pi(\id_{K^c})$ by transforming it into a more accessible form using a result of Beuzart-Plessis (\Cref{thm:BP_I-J}) and our transfer statements in \Cref{section_fl}. In \Cref{section_applications} we show the existence of newforms (\Cref{exist_newform}) and calculate the local factors in the Ichino--Ikeda conjecture (\Cref{cor_I-I}) when newforms are unique up to scalar. Below is a diagram that summarizes the main structure of our proofs.

\begin{center}
\begin{tikzpicture}
\tikzstyle{block} = [rectangle, draw, text centered, minimum height=2em]
\node (thmnf){\Cref{exist_newform}};
\node [below=0.5cm of thmnf](thm1){\Cref{thm:mainthm}};
\node [coordinate, below right=0.4cm of thmnf](sp0){};
\node [coordinate, right=0.5cm of sp0](sp1){};
\node [above right =0.4cm of sp1](thm3){\Cref{hom_fl}};
\node [below=0.5cm of thm3](lem10){\Cref{lemma_calc_I}};

\draw [-stealth](sp0) |- (thm1);
\draw [-stealth](sp0) |- (thmnf);
\draw (sp0) -- (sp1);
\draw (thm3) -| (sp1);
\draw (lem10) -| (sp1);

\node [right=0.5cm of thm3](lem4){\Cref{inhom_fl}};
\draw [-stealth](lem4) -- (thm3);

\node [right=0.5cm of lem4](lem9){\Cref{lie_alg_fl}};
\draw [-stealth](lem9) -- (lem4);
\end{tikzpicture}
\end{center}

Terminologies and notations that are not defined above will be introduced in subsequent sections.
Since $E/F$ is assumed to be unramified throughout the paper, we will not repeatedly state this assumption, but note that some definitions and results only hold in this case, and not in general. 

\noindent \textbf{Acknowledgment.} The author is grateful to her advisor Wei Zhang for suggesting this problem and offering guidance throughout the project. The author would also like to thank Murilo Corato Zanarella, Zhiwei Yun, Hiraku Atobe, Atsushi Ichino, and Hang Xue for their helpful comments.

\section{Notations and measures}
\label{notation_n_measure}
\subsection{Notations}
\label{section_notations}
In this section, we introduce some notations that will be used throughout the paper. 

Let \index{En@$E^n$} $E^n$ denote the vector space of $n$-dimensional column vectors over $E$ and \index{En@$E_n$} $E_n$ the vector space of $n$-dimensional row vectors over $E$, similarly for $F$.
For an $m$-dimensional hermitian space $V$ over $E$ with hermitian matrix $J$ (with respect to a chosen basis), the corresponding unitary group is defined to be  \index{UV@$U(V)$}
\[U(V)=\{g \in \Res_{E/F}GL_m: \bar{g}^t J g=J\}.\]
Denote by \index{uV@$\ul(V)$} $\ul(V)$ the Lie algebra of $U(V)$, then
\[\ul(V)=\{X\in \Res_{E/F}\Mat_{m}: JX+\br{X}^tJ=0\}.\]
For all $\xi\in E$ with $\xi \bar{\xi}=1$, we define the Cayley map \index{cayx@$\cay_\xi$} $\cay_\xi$ associated to $\xi$ as follows:
\begin{align*}
\cay_\xi: \ul(V) &\to U(V)\\
X &\mapsto \xi \frac{1+X}{1-X}.
\end{align*}
It is straightforward to check that $\cay_\xi$ defines a birational isomorphism from $\ul(V)$ to $U(V)$ and an $U(W)(F)$-equivariant (under conjugation action) bijection 
\[{\{X\in \ul(V)(F): \det (X-1)\neq 0\}} \cong \{g\in U(V)(F): \det(g+\xi)\neq 0\}.\]

For $m>0$, we define a variety \index{Sm@$S_m$} $S_m$ over $F$ by 
\[S_m=\{s\in \Res_{E/F} \GL_m: s\bar{s}=1\}\]
and its ``Lie algebra'' \index{sm@$\s_m$}
\[\s_m:=\{s \in \Res_{E/F} \Mat_m: s+\bar{s}=0\}.\]
For $\xi\in E$ with $\xi \bar{\xi}=1$, we denote the map $X \mapsto \xi \frac{1+X}{1-X}$ on $\s_m$ also by $\cay_\xi$. Then $\cay_\xi$ defines a birational isomorphism from $\s_m$ to $S_m$ and a $\GL_{n}(F)$-equivariant bijection 
\[{\{Y\in \s_{n+1}(F): \det (Y-1)\neq 0\}} \cong \{s\in S_{n+1}(F): \det(s+\xi)\neq 0\}.\]

When $\xi=1$, we write \index{cay@$\cay$} $\cay$ for $\cay_\xi$ for simplicity.
Let \index{nu@$\nu$} $\nu:E^\times \to \Z$ be the valuation on $E$, normalized so that $\nu(\vp)=1$.  
Define the absolute values on $E^\times$ (resp. $F^\times$) to be $|\cdot |=q_E^{\nu(\cdot)}$ (resp. $|\cdot|=q_F^{\nu(\cdot)}$).
We abuse the notation and write $\nu$ for $\nu\circ \det$ as well.
Define \index{eta@$\eta_{E/F}$} $\eta_{E/F}$ to be the quadratic character on $F^\times$ associated to the extension $E/F$ by local class field theory. 
Since $E/F$ is unramified, we have $\eta_{E/F}(a)=(-1)^{\nu(a)}$ for all $a\in F^\times$. Denote by \index{eta@$\wtil{\eta}_{E/F}$} $\wtil{\eta}_{E/F}$ the character on $E^\times$ given by $a \mapsto (-1)^{\nu(a)}$. 
We abuse the notation and denote $\eta_{E/F}\circ \det$  by $\eta_{E/F}$ as well. 
Fix an additive character \index{psi@$\psi$} $\psi: F\to \C^\times$ with conductor $\roi_F$ (meaning that $\psi$ is trivial on $\roi_F$ and nontrivial on $\vp^{-1}\roi_F$) and define \index{psie@$\psi_E$} $\psi_E: E\to \C^\times$ by $\psi_E(a)=\psi(\frac{1}{2} \tr_{E/F}(a))$.

Sometimes we will abuse the notation and identify algebraic groups with their group of points when this does not cause confusions.

\subsection{Measures}
\label{section_measures}
We choose the normalization of Haar measures as in \cite{beuzart-plessis_2021}. Along the way, we will calculate the volumes of some groups that come up later. Set the measures on $F$ and $E$ to be so that $\vol(\roi_F)= \vol(\roi_E)=1$. Define the measure on $\GL_m(F)$ to be
\[dg=\zeta_F(1)\frac{\prod_{ij}dg_{ij}}{|\det g|_F^m}\]
for $m\geq 1$ and similarly define the measure on $\GL_m(E)$.
Let \index{Nm@$N_m$} $N_m$ denote the standard unipotent subgroup of $\GL_m$ consisting of upper triangular matrices whose diagonal entries are 1.
Define the measure on $N_m(F)$ to be 
\[du=\prod_{1\leq i<j\leq m}du_{ij}\]
for $m\geq 1$ and similarly define $N_m(E)$. 
Note that we have a short exact sequence
\[1\to 1+\vp\Mat_m(\roi_F) \xrightarrow{\text{embed}} \GL_m(\roi_F)\xrightarrow{\text{mod }\vp} \GL_m(\F_{q_F}) \to 1.\]
Since determinants of elements in $\GL_m(\roi_F)$ are in $\roi_F^\times$, we can ignore the denominator in the Haar measure when integrating over (subgroups of) $\GL_m(\roi_F)$. Thus, 
\begin{align}
\label{vol_GL_OF}
\vol(GL_m(\roi_F)) &= |\GL_m(\F_{q_F})|\vol(1+\vp M_m(\roi_F))\nonumber \\ 
&= |\GL_m(\F_{q_F})|\zeta_F(1) q_F^{-m^2}\\
& =\zeta_F(1)\prod_{i=1}^m (1-q_F^{-i})  \nonumber.
\end{align}
Similarly,
\begin{equation}
\label{vol_GL_OE}
\vol(GL_m(\roi_E))= \zeta_E(1)\prod_{i=1}^m (1-q_E^{-i}).
\end{equation}
Let \index{Kctilde@$\wtil{\bm{K}}^c_{n+1}$} \index{Kc@$\bm{K}^c_{n+1}$} 
\[
\wtil{\bm{K}}^c_{n+1}=\{x\in \Mat_{n+1}(\roi_E)| x\equiv \begin{blockarray}{ccc}
 & n &  1  \\
\begin{block}{c(cc)}
  n & * & 0 \\
  1 & * &  * \\
\end{block}
\end{blockarray} \mod \vp^c \},
\]
and
\[
\bm{K}^c_{n+1}=\{x\in \Mat_{n+1}(\roi_E)| x\equiv \begin{blockarray}{ccc}
 & n &  1  \\
\begin{block}{c(cc)}
  n & * & 0 \\
  1 & * &  1 \\
\end{block}
\end{blockarray} \mod \vp^c \}.\\
\]
Define 
\index{Kn_prime@$K_n'$} \index{Kn1_prime@$K_{n+1}^{\prime, c}$} \index{K_prime@$K^{\prime, c}$} \index{Ktn_prime@$\wtil{K}_n'$} \index{Ktn1_prime@$\wtil{K}^{\prime,c}_{n+1}$} 
\begin{align*}
K_{n+1}^{\prime, c}=&\bm{K}^c_{n+1}\cap \GL_{n+1}(\roi_E), \quad K'_n=\GL_n(\roi_E), \quad K^{\prime, c}=K'_n \times K_{n+1}^{\prime, c},\\
\wtil{K}^{\prime,c}_{n+1}=&\wtil{\bm{K}}^c_{n+1}\cap \GL_{n+1}(\roi_E), \quad \wtil{K}'_n=\GL_n(\roi_E).
\end{align*}
Let $\wtil{K}''_{n+1}$ be the kernel of the surjective map 
\begin{align*}
\wtil{K}_{n+1}^{\prime, c} &\to \GL_n(\F_{q_E}) \times \F_{q_E}^\times\\
\mat{X&y\\z&w} &\mapsto (X \text{ mod }\vp, w \text{ mod } \vp),
\end{align*}
then we have
\[\wtil{K}''_{n+1}=\{\mat{X&y\\z&w} \in \GL_{n+1}(\roi_E)| X\in 1+\vp\Mat_{n \times n} \roi_E, y\in (\vp^c \roi_E)^n, z\in (\roi_E)_n, w\in 1+\vp\roi_E\}\]
and
\[\vol(\wtil{K}''_{n+1})= \zeta_E(1) q_E^{-n^2-cn-1}.\]
Hence 
\[\vol(\wtil{K}^{\prime,c}_{n+1})= \vol(\wtil{K}''_{n+1}) |\GL_n(\F_{q_E})| |\F_{q_E}^\times| = \zeta_E(1) q_E^{-cn-1}(q_E-1)\prod_{i=1}^{n} (1-q_E^{-i})\]
and
\begin{equation}
\label{vol_kprime}
\vol(K^{\prime,c}_{n+1})=(\#(\roi_E/\vp^c \roi_E)^\times)^{-1} \vol(\wtil{K}^{\prime,c}_{n+1})
= \zeta_E(1) q_E^{-c(n+1)}\prod_{i=1}^{n} (1-q_E^{-i}).
\end{equation}
Similarly,
\begin{equation}
\label{vol_GL_KF}
\vol(\bm{K}^c_{n+1}\cap \GL_{n+1}(\roi_F))=
\zeta_F(1) q_F^{-c(n+1)}\prod_{i=1}^{n} (1-q_F^{-i}).
\end{equation}

Let $W, V, \Lambda_n$ be as in \Cref{section_intro}. We choose an $\roi_E$-basis \index{BW@$\mathcal{B}_W$} $\mathcal{B}_W$ of $\Lambda_n$ so that the hermitian form on $W$ is the identity matrix $I_n$ with respect to this basis. 
Append $\vp^{\fl{c/2}}e$ to this basis to get a basis \index{BV@$\mathcal{B}_V$} $\mathcal{B}_V$ of $V$, then when $2|c-\epsilon$ the hermitian form on $V$ is \index{J@$J$} $J=\mat{I_n&\\&\vp^c}$ with respect to $\mathcal{B}_V$. We fix this choice of bases throughout the paper.

Define a pairing on $\ul(V)(F)$ by 
\[\langle X,Y \rangle = \tr(XY), \quad X,Y \in \ul(V)(F).\]
Then there is a Fourier transform on $C^\infty_c(\ul(V)(F))$ given by 
\[\hat{\phi}(X)=\int_{\ul(V)(F)}\phi(Y)\psi(\pa{X,Y})dX, \quad \phi \in C^\infty_c(\ul(V)(F)),\]
and we define the measure on $\ul(V)(F)$ to be the self-dual one, i.e., the one that satisfies  
\[\hat{\hat{\phi}}(X)=\phi(-X)\]
for all $\phi\in C^\infty_c(\ul(V)(F))$. Define the measure on $\s(F)$ similarly.

Choose the Haar measure on $U(V)(F)$ to be the one such that the Jacobian of the Cayley map $\cay$ at 0 is $L(1,\eta_{E/F})$. 
We now calculate the volume of $U(V)(\roi_F)$.

Note that for $A=\mat{X & y \\ z & w} \in \ul(V)(F)$ with $X\in \Mat_{n \times n}(E), y\in E^n, z\in E_n, w\in E$, we have $X+\br{X}^t=0, w+\bar{w}=0, y= \vp^c \bar{z}^t$. Hence under the trace pairing the dual lattice of $\ul(V)(\roi_F)$ is 
\[\widehat{\ul(V)(\roi_F)}=\{A=\mat{X & y \\ z & w} \in \ul(V)(F)| X\in \Mat_{n \times n}(\roi_E), y\in \roi_E^n, w\in \roi_E\}\]
and 
\[\widehat{\ul(V)(\roi_F)}/ \ul(V)(\roi_F)=(\roi_E/\vp^c \roi_E)^n.\]

Plug $\phi=\id_{\ul(V)(\roi_F)}$ into the equation $\hat{\hat{\phi}}(X)=\phi(-X)$, we get 
\[\vol(\ul(V)(\roi_F))=|\widehat{\ul(V)(\roi_F)}/ \ul(V)(\roi_F)|^{1/2}=q_F^{-cn}.\]

Consider the subset $\mathfrak{k}_0$ of $\ul(V)(\roi_F)$ defined by
\[\mathfrak{k}_0=\{\mat{X&y\\z&w} \in \ul(V)(\roi_F)|X\in \Mat_{n \times n} (\vp \roi_E), w\in \vp \roi_E\}\]
and the subset $K_0$ of $U(V)(\roi_F)$ defined by 
\[K_0=\{\mat{X&y\\z&w}\in U(V)(\roi_F)| X\in 1+\Mat_{n \times n} (\vp \roi_E), w\in 1+\vp \roi_E\}.\]
For $m\geq 1$ let $U(m)$ be the unique unitary group of an $m$-dimensional hermitian space over $\F_{q_F}$ and let $\ul(m)$ be its Lie algebra. Note that
\[\ul(V)(\roi_F)/\mathfrak{k}_0=\ul(n)(\F_{q_F}) \ul(1)(\F_{q_F}),\] 
so
\[\vol(\mathfrak{k}_0)=\vol(\ul(V)(\roi_F))|\ul(n)(\F_{q_F})|^{-1} |\ul(1)(\F_{q_F})|^{-1} = q_F^{-cn-n^2-1}.\]

The Cayley map induces a bijection from $\mathfrak{k}$ to $K_0$.
So with respect to the measure we chose,
\[\vol(K_0)=L(1,\eta_{E/F})\vol(\mathfrak{k}_0)=L(1,\eta_{E/F}) q_F^{-cn-n^2-1}.\]

Since $K_0$ is the kernel of the surjective map 
\begin{align*}
\rho: U(V)(\roi_F) &\to U(n)(\F_{q_F}) \times U(1)(\F_{q_F})\\
\begin{blockarray}{ccc}
  & n &  1  \\
 \begin{block}{c(cc)}
   n & X & y \\
   1 & z &  w \\
 \end{block} 
 \end{blockarray} &\mapsto (X \text{ mod } \vp, w \text{ mod } \vp),
\end{align*}
we have
\[\vol(U(V)(\roi_F))=\vol(K_0)\cdot |U_n(\F_{q_F})| \cdot |U_1(\F_{q_F})|= L(1,\eta_{E/F}) q_F^{-cn} (1+q_F^{-1}) \prod_{i=1}^n (1-(-q_F)^{-i}).\]

Similarly, $\vol(\ul(W)(\roi_F))=1$ and 
\begin{align}
\label{vol_UW_OF}
\nonumber \vol(U(W)(\roi_F))&= L(1,\eta_{E/F}) \vol(\ul(W)(\roi_F) \frac{|U_n(\F_{q_F})|}{|\ul(n)(\F_{q_F})|}\\
&= L(1,\eta_{E/F}) \prod_{i=1}^n (1-(-q_F)^{-i}).
\end{align}

\section{An explicit transfer statement}
\label{section_fl}
\subsection{Orbital integrals and smooth transfers}
\label{section_orb_int}

Let \index{W01@$W_0, W_1$} $W_0$ (resp. $W_1$) be an $n$-dimensional hermitian space over $E$ with discriminant 1 (resp $-1$), and let \index{V01@$V_0, V_1$} $V_i=W_i\oplus^\perp \pa{u_i}$ for some $\pa{u_0, u_0}= \pa{u_1, u_1}$. Then the discriminant of $V_i$ is $(-1)^{\nu(\pa{u_i,u_i})+i}$. Let \index{G01@$G_0, G_1$} $G_i=U(W_i) \times U(V_i)$ and consider \index{H01@$H_0, H_1$} $H_i=U(W_i)$ as a subgroup of $G_i$ by the diagonal embedding.

Let \index{G_prime@$G'$} $G'=\Res_{E/F}(\GL_{n} \times \GL_{n+1})$ with subgroups \index{H1prime@$H'_1$} \index{H2prime@$H'_2$} $H'_1=\Res_{E/F}GL_{n}$ and $H'_2=\GL_{n,F} \times \GL_{n+1, F}$, where $H'_1$ is regarded as a subgroup of $G'$ through the diagonal embedding. We want to compare orbital integrals on $G_i(F)$ and $G'(F)$. In the rest of \Cref{section_orb_int} the index $i$ will always mean $i=0,1$.

For a reductive group $G$ over $F$ with an action on a variety $X$ over $F$, we say a point $x$ in $X$ is regular semisimple under the action of $G$ if its orbit $Gx$ is Zariski closed in $X$ and of maximal dimension (the second condition is equivalent to the condition that its stabilizer is trivial). Denote by $X(F)^\rs$ the subgroup of regular semisimple elements in $X(F)$ under the action of $G$ and denote by $X(F)^\rs/G(F)$ the set of orbits of $X(F)^\rs$ under the action of $G(F)$.

There is an action of $H_i(F) \times H_i(F)$ on $G_i(F)$ given by $(h_1, h_2): g \mapsto h_1gh_2$. 
With respect to this action, the orbital integrals of Schwartz functions on $G_i(F)$ are given by \index{Orb@$\Orb(x,f)$}
\[\Orb(x,f)=\int_{H_i(F)\times H_i(F)} f(h_1xh_2)dh_1dh_2,\quad x\in G_i(F)^\rs, f\in C^\infty_c(G_i(F)).\]

With the embedding $W_i\subset V_i$, $H_i(F)$ can be viewed naturally as a subgroup of $U(V_i)(F)$.
The orbital integrals with respect to the conjugation action of $H_i(F)$ on $U(V_i)(F)$ are defined to be
\[\Orb(x,f)=\int_{H_i(F)}f(hxh^{-1})dh, \quad x\in U(V_i)(F)^\rs, f\in C^\infty_c(U(V_i)(F)).\]

Similarly consider the conjugation action of $H_i(F)$ on $\ul(V_i)(F)$. The corresponding orbital integrals are
\[\Orb(X,f)=\int_{H_i(F)}f(hXh^{-1})dh, \quad X\in \ul(V_i)(F)^\rs, f\in C^\infty_c(\ul(V_i)(F)).\]
 
The map $(x_1, x_2) \mapsto (x_1^{-1}x_2)$ gives a surjection $G_i(F)^\rs \to U(V_i)(F)^\rs$. For $f\in C^\infty_c(G_i(F))$ and $x\in U(V_i)(F)$ let \index{ft@$\tilde{f}$}
\begin{equation}
\label{f_tilde}
\tilde{f}(x)=\int_{H_i(F)} f(h(1,x))dh,
\end{equation}
then $\tilde{f}\in C^\infty_c(U(V_i)(F))$ and
\begin{equation}
\label{u_hom_to_inhom}
\Orb((x_1, x_2),f)=\Orb(x_1^{-1}x_2,\tilde{f}), \quad (x_1, x_2)\in G_i(F)^\rs.
\end{equation}
For $\xi\in E$ with $\xi\bar{\xi}=1$ and $f\in C^\infty_c(U(V_i)(F))$, define \index{fcay@$f_\xi$}
\begin{equation}
f_\xi(X)=\begin{cases}
    f(\cay_\xi(X)) & \text{ if } \det (X-1) \neq 0,\\
    0 & \text{ otherwise},
\end{cases}
\end{equation}
then $f_\xi \in C^\infty_c(\ul(V_i)(F))$ and for $g\in U(V_i)(F)^\rs$ satisfying $\det (g+\xi)\neq 0$ we have $\cay_\xi^{-1}(g) \in \ul(V_i)(F)^\rs$ and
\[\Orb(g, f)=\Orb(\cay_\xi^{-1}(g), f_\xi).\]

Now we turn to the general linear group side. 
We define a character on $H'_2(F) = \GL_{n}(F) \times \GL_{n+1}(F)$ by
\[(g_1, g_2)\mapsto \eta_{E/F}(g_1)^{n+1}\eta_{E/F}(g_2)^{n},\quad (g_1, g_2)\in H'_2(F)\]
and denote it again by \index{eta@$\eta_{E/F}$} $\eta_{E/F}$. It will be clear from context if $\eta_{E/F}$ is a character of $F^\times, \GL_n(F), \GL_{n+1}(F)$, or $H_2(F)$, so this should not cause confusion.

Let $H'_1(F) \times H'_2(F)$ act on $G'$ from left and right. 
The orbital integral with respect to this action is given by \index{Orb@$\Orb(x,f)$}
\[\Orb(x,f)=\int_{H'_1(F) \times H'_2(F)} f(h_1^{-1}xh_2) \eta_{E/F}(h_2) dh_1 dh_2,\quad x\in (G')(F)^\rs, f\in C^\infty_c (G'(F)).\]

We view $\GL_n(F)$ as a subgroup of $\GL_{n+1}(F)$ using the embedding $g \mapsto \mat{g&\\&1}$ and let $\GL_n(F)$ act on $S_{n+1}(F)$ by conjugation. The corresponding orbital integrals are defined by
\[\Orb(x,f)=\int_{GL_{n}(F)} f(h^{-1}xh)\eta_{E/F}(h)dh, \quad x\in S_{n+1}(F)^\rs, f\in C^\infty_c(S_{n+1}(F)).\]

Similarly, let $\GL_n(F)$ act on $\s_{n+1}(F)$ by conjugation. The corresponding orbital integrals are
\[\Orb(X,f)= \int_{\GL_{n}(F)} f(h^{-1}Xh) \eta_{E/F}(h) dh, \quad X\in \s_{n+1}(F)^\rs, f\in C^\infty_c(\s_{n+1}(F)).\]

Note that the map \index{r@$r$} $r: \GL_{n+1}(E) \to S_{n+1}(F)$ given by $g \mapsto g\br{g}^{-1}$ is surjective.
For $(x_1, x_2)\in (G')(F)^\rs$ and $f\in C^\infty_c(G'(F))$, we have $r(x_1^{-1}x_2) \in S_{n+1}(F)^\rs$ and
\begin{equation}
\label{s_hom_to_inhom}
\Orb((x_1, x_2),f)= 
\tilde{\eta}_{E/F}(x_1^{-1}x_2)^{n}\Orb(r(x_1^{-1}x_2), \tilde{f}),
\end{equation}
where $\tilde{f}\in C^\infty_c(S_{n+1}(F))$ is defined by \index{ft@$\tilde{f}$}
\begin{equation}
\label{f_tilde_s}
\tilde{f}(x)=
\int_{H'_1(F)} \int_{\GL_{n+1}(F)} f(h_1(1,gh_2)) \tilde{\eta}_{E/F}(gh_2)^n dh_2dh_1 
\end{equation}
for $x=r(g)$.
For $s\in S_{n+1}(F)^\rs$, $\xi\in E$ satisfying $\xi\bar{\xi}=1$ and $\det (s+\xi)\neq 0$, and $f\in C^\infty_c(S_{n+1}(F))$ we have $\cay_\xi^{-1}(s) \in \s_{n+1}(F)^\rs$ and
\[\Orb(s, f)=\Orb(\cay_\xi^{-1}(s), f_\xi),\] 
with $f_\xi$ defined as in \Cref{f_tilde} in the case of unitary groups above.

Next, we recall the notions of matching and smooth transfer (cf. \cite[Section 3.4]{beuzart-plessis_2021}). 
A transfer factor for $G'$ is a function \index{OmegaG@$\Omega_{G'}$} $\Omega_{G'}: G'(F)^\rs\to \C^\times$ such that 
\[\Omega_{G'}(h_1^{-1} x h_2)=\eta_{E/F}(h_2)\Omega_{G'}(x), \quad x\in G'(F)^\rs, h_1\in H_1(F), h_2\in H_2(F).\]
A transfer factor for $S_{n+1}$ is a function \index{OmegaS@$\Omega_S$} $\Omega_S: S_{n+1}(F)^\rs \to \C^\times$ such that
\[\Omega_S(h^{-1}xh)=\eta_{E/F}(h)\Omega_S(x), \quad x\in S_{n+1}(F)^\rs, h\in GL_{n}(F).\]
A transfer factor for $\s_{n+1}$ is a function \index{omega@$\omega$} $\omega: \s_{n+1}(F)^\rs \to \C^\times$ such that
\[\omega(h^{-1}Xh)=\eta_{E/F}(h)\omega(X), \quad X\in \s_{n+1}(F)^\rs, h\in GL_{n}(F).\]
Let \index{e0@$e_0^*$} $e_0^*=(0\ \cdots\ 0\ 1) \in E_{n+1}$. Following \cite{beuzart-plessis_2021} and \cite{Yun}, we choose the transfer factors to be \index{OmegaG@$\Omega_{G'}$} \index{OmegaS@$\Omega_S$} \index{omega@$\omega$}
\begin{align*}
\omega(X)&= \tilde{\eta}_{E/F}(\det (e_0^*, e_0^* X, \cdots, e_0^* X^n)),\\
\Omega_S(x)&=\tilde{\eta}_{E/F}(\det (e_0^*, e_0^* x, \cdots, e_0^* x^n)),\\
\Omega_{G'}((x_1, x_2))&=
\tilde{\eta}_{E/F}(x_1^{-1}x_2)^n \Omega_S(\nu(x_1^{-1} x_2)).
\end{align*}

We say that $X\in \ul(V_i)(F)^\rs$ matches $Y\in \s_{n+1}(F)^\rs$ if they are $\GL_{n}(E)$-conjugate in $\gl_{n+1}(E)$. The matching relation induces a bijection \cite{J-R}
\[\s_{n+1}(F)^\rs/\GL_n(F)=\ul(V_0)(F)^\rs/U(W_0)(F) \sqcup \ul({V_1})(F)^\rs/U(W_1)(F).\] 
We say $f_i\in C^\infty_c(\ul(V_i)(F))$ matches (or is the smooth transfer of) $f' \in C^\infty_c(\s_{n+1}(F))$  if 
\[\Orb(X,f_i)=\omega(Y)\Orb(Y,f')\]
for all matching $X_i\in \ul(V_i)(F)^\rs$ and $Y\in \s_{n+1}(F)^\rs$. We say that a pair $(f_0, f_1)\in C^\infty_c(\ul(V_0)(F)) \times C^\infty_c(\ul(V_1)(F))$ matches $f'\in C^\infty_c(\s_{n+1}(F))$ if both $f_0$ and $f_1$ match $f'$.

We say that $x\in U(V_i)(F)^\rs$ matches $s\in S_{n+1}(F)^\rs$ if they are $\GL_{n}(E)$-conjugate and similarly this gives a bijection \cite[Section 2.1]{Zhang_AFL}
\[S_{n+1}(F)^\rs/\GL_n(F)=U(V_0)(F)^\rs/U(W_0)(F) \sqcup U({V_1})(F)^\rs/U(W_1)(F).\] 
We say $f_i\in C^\infty_c(U(V_i)(F))$ matches $f' \in C^\infty_c(S_{n+1}(F))$ if 
\[\Orb(g_i,f_i)=\Omega_S(s)\Orb(s,f')\]
whenever $g_i\in U(V_i)^\rs(F)$ and $s\in S_{n+1}(F)$ match. Similarly, we say that $(f_0, f_1)\in C^\infty_c(U(V_0)(F)) \times C^\infty_c(U(V_1)(F))$ and $f' \in C^\infty_c(S_{n+1}(F))$ match if both $f_0$ and $f_1$ match $f'$.

We say that $(g_1, g_2)\in G_i(F)^\rs$ matches $(g'_1, g'_2)\in G'(F)^\rs$ if  $g_1^{-1}g_2 \in U(V_i)(F)^\rs$ and $r({g'_1}^{-1}g_2') \in  S_{n+1}(F)^\rs$ match in the sense above. Similarly, we have
\[G'(F)^\rs/(H_1'(F)\times H'_2(F))=G_0(F)^\rs/H_0(F) \sqcup G_1(F)^\rs/H_1(F).\] 
We say $f_i\in C^\infty_c(G_i(F))$ matches $f'\in C^\infty_c(G'(F))$ if 
\[\Orb((g_1, g_2),f)=\Omega_{G'}((g_1',g_2'))\Orb((g_1', g_2'),f')\]
for all matching $(g_1, g_2)\in G_i(F)^\rs$ and $(g'_1, g'_2)\in G'(F)^\rs$. 
Then $(f_0, f_1) \in  C^\infty_c(G_0(F)) \times  C^\infty_c(G_1(F))$ and $f'\in C^\infty_c(G'(F))$ said to match if both $f_0$ and $f_1$ match $f'$.\\

Recall the definitions of some local integrals from \cite{beuzart-plessis_2021}. Let $\Pi=\Pi_n \boxtimes \Pi_{n+1}$ be a representation of $G'(F)=\GL_n(E) \times \GL_{n+1}(E)$ such that $\Pi_n, \Pi_{n+1}$ are generic tempered unitary representations of $GL_n(E)$ and $GL_{n+1}(E)$ respectively. 

We fix a \index{tau@$\tau$} $\tau\in \roi_E^\times$ such that $\tr_{E/F}(\tau)=0$. 
Recall that $\psi:F\to \C^\times$ is a nontrivial additive character of $F$ with conductor $\roi_F$ and $\psi_E: E\to \C^\times$ is defined by $\psi_E(a)=\psi(\frac{1}{2} \tr_{E/F}(a))$.  
Let \index{W@$\W(\pi,\psi)$} $\W(\Pi_n, \br{\psi}_E)$ and $\W(\Pi_{n+1},\psi_E)$ be the Whittaker models of $\Pi_n$ and $\Pi_{n+1}$ with respect to $\br{\psi}_E$ and $\psi_E$ respectively, and let $\W(\Pi)=\W(\Pi_n, \br{\psi}_E) \otimes \W(\Pi_{n+1},\psi_E)$.

The Flicker--Rallis integrals \index{beta@$\beta_n, \beta_{n+1}, \beta$} $\beta_n: \W(\Pi_n, \br{\psi}_E) \to \C$ and $\beta_{n+1}: \W(\Pi_{n+1},\psi_E) \to \C$ are defined by
\[\beta_k(W_k)=\int_{N_{k-1}(F)\backslash GL_{k-1}(F)} W_k(\epsilon_k(\tau) g_{k-1}) \eta_{E/F}(\det g_{k-1})^{k-1} dg_{k-1}\]
for $k=n,n+1$, where \index{epsilontau@$\epsilon_k(\tau)$} $\epsilon_k(\tau)=\diag(\tau^{k-1}, \tau^{k-2}, \cdots, 1)$. Define
\[\beta= \beta_n \otimes \beta_{n+1}.\]

The Rankin--Selberg integral \index{lambda@$\lambda$} $\lambda:\W(\Pi)\to \C$ is defined by 
\[\lambda(s,W_n \otimes W_{n+1})= \int_{N_n(E)\setminus \GL_n(E)} W_n(g_n)W_{n+1}(g_{n})|\det g_n|_E^s\ dg_n.\]
We write $\lambda(W_n \otimes W_{n+1})$ for $\lambda(0,W_n \otimes W_{n+1})$ as this is the only case concerned.

We can define an inner product $\theta$ on $\W$ as follows.
Define inner products $\theta_n$ and $\theta_{n+1}$ on $\W(\Pi_n, \br{\psi}_E)$ and $\W(\Pi_{n+1},\psi_E)$ respectively by \index{theta@$\theta_n, \theta_{n+1}, \theta$}
\[\theta_k(W_k, W_k')=\int_{N_{k-1}(E) \setminus \GL_{k-1}(E)} W_k(g_{k-1}) \br{W'_k(g_{k-1})} dg_{k-1},\]
where $k=n,n+1$. Let $\theta=\theta_n \otimes \theta_{n+1}$. The integrals $\beta_k$, $\theta_k$ and $\lambda$ converge as explained in Section 3.2 of \cite{beuzart-plessis_2021}.

For any $f\in C^\infty_c(G'(F))$, choose a compact open subgroup $K_f$ of $G'(F)$ such that $f$ is bi-$K_f$-invariant. Let \index{Ipi@$I_\pi$}
\[I_\Pi(f)=\sum_{W\in ON(\Pi^{K_f})} \lambda(\Pi(f)W) \br{\beta(W)},\]
where $ON(\Pi^{K_f})$ is an orthonormal basis of $\Pi^{K_f}$ with respect to the inner product $\theta$. The sum $I_\Pi$ is absolutely convergent and does not depend on the choice of $ON(\Pi^{K_f})$.

The integrals $I_\Pi$ and $J_\pi$ are related in the following way. 
\begin{thm}
\label{thm:BP_I-J}
Let $\pi=\pi_n \times \pi_{n+1}$ be an irreducible tempered unitary $H_i$-distinguished representation of $G_i(F)$, then for matching $f_i\in C^\infty_c(G_i(F))$ and $f' \in C^\infty_c(G'(F))$ we have
\[I_{BC(\pi)}(f')=\omega_{BC(\pi_n)}(\tau)L(1, \eta_{E/F})^{-1}J_\pi(f_i),\]
where $\omega_{BC(\pi_n)}$ is the central character of $BC(\pi_n)$.
\end{thm}

\begin{rmk}
Note that $\omega_{BC(\pi_n)}(\tau)= 1$ when $\pi_n$ is unramified. 
\end{rmk}

Zhang \cite{Zhang14} proved this result for supercuspidal representations and later Beuzart-Plessis \cite[Theorem 3.5.7]{beuzart-plessis_2021} proved it in the general case.

Note that, with notations as in \Cref{section_intro}, to calculate $J_\pi(\id_{K^c})$ it is enough to find $f'\in C^\infty_c(G'(F))$ that matches $\id_{K^c}$ and calculate $I_{BC(\pi)}(f')$, which is what we do in the rest of this section and \Cref{section_calc_I} respectively.

\subsection{An explicit transfer statement}
The key result we will need is Jacquet--Rallis's Fundamental lemma \cite{J-R}, proved for fields with large enough residue field characteristic by Gordon and Yun \cite{Yun} and in full generality by Beuzart-Plessis \cite{BP_fl-proof}.  In \cite{Zhang_afl21}, Wei Zhang gave another proof via global method for large $p$ in parallel to his proof of the arithmetic fundamental lemma and Zhiyu Zhang \cite{zzy} proved both the usual and arithmetic fundamental lemma for all odd $p$.

\begin{thm}[Jacquet--Rallis's Fundamental Lemma]
\label{JR_fl}
Let $V_0, V_1, W_0, W_1$ be as in \Cref{section_orb_int} and suppose that the hermitian form on $W_0$ is given by the identity matrix $I_n$ and $\pa{u_i, u_i}=1$ for $i=0,1$. 
Then the pair $(\vol(U(W_0)(\roi_F))^{-1} \id_{\ul(V_0) (\roi_F)},0) \in C^\infty_c(\ul(V_0)(F)) \times C^\infty_c(\ul(V_1)(F))$ matches $\vol(\GL_n(\roi_F))^{-1} \id_{\s_{n+1}(\roi_F)} \in C^\infty_c(\s_{n+1}(F))$.
\end{thm}

In the rest of \Cref{section_fl}, we fix a non-negative integer $c$. Let $W_0$ be the $n$-dimensional hermitian space over $E$ with hermitian form $I_n$. Let $V_0=W_0\oplus^\perp \pa{u_0}$, where $\pa{u_0,u_0}=\vp^c$. 
Let $W_1$ be an $n$-dimensional hermitian space over $E$ with discriminant $-1$, then choose $u_1, V_1$ and define $G_0, G_1, H_0, H_1$ accordingly as in \Cref{section_orb_int}.
Let \index{Kn@$K_n$} $K_n=U(W_0)(\roi_F)$, \index{Kn1@$K^c_{n+1}$}
$K^c_{n+1}=U(V_0)(\roi_F)\cap \bm{K}^c_{n+1}$, and $K^c=K_n \times K^c_{n+1}$. Note that when we set $\Lambda_n=\roi_E^n \subset W_0$ and $\Lambda_{n+1}=\roi_E^n \oplus \roi_E u_0 \subset V_0$, definitions of $K_n, K^c_{n+1}, K^c$ here coincide with those given in \Cref{section_intro}.

The main result we prove in this section is the following transfer statement. 
\begin{thm}[homogeneous version]
\label{hom_fl}
With $K^{\prime,c}$ as in \Cref{notation_n_measure}, the pair $(\id_{K^c},0) \in C^\infty_c(G_0(F)) \times  C^\infty_c(G_1(F))$ matches $c_1 \id_{K^{\prime,c}} \in C^\infty_c(G'(F))$
where \index{c1@$c_1$}
\begin{align*}
c_1&=\frac{\vol(K_n)^2}{\vol(\GL_n(\roi_F))\vol(K'_n)\vol(K^{\prime, c}_{n+1}\cap \GL_{n+1}(\roi_F))}\\
&=L(1,\eta_{E/F})^2\zeta_F(1)^{-2}\zeta_E(1)^{-1}q_F^{c(n+1)}
\prod_{i=1}^n \frac{(1-(-q_F)^{-i})^2}{(1-q_F^{-i})^{3} (1+q_F^{-i})}.
\end{align*}
\end{thm}

Rapoport, Smithling, and Zhang \cite{RSZ} proved a similar result when $c=1$. Their approach is easily generalizable to the case of $c>1$, so we will prove \Cref{hom_fl} following their method. 
 
Let \index{Knt@$\wtil{K}^c_{n+1}$} $\widetilde{K}^c_{n+1}=\wtil{\bm{K}}^c_{n+1} \cap U(V_0)(F)$ and \index{Knt@$\wtil{K}_n$} $\widetilde{K}_{n}={\wtil{K}^c_{n+1}\cap U(W_0)(F)} =U(W_0)(\roi_F)=K_{n}$.
Let \index{Kst@$\wtil{K}^{\prime, c}_S$} $\wtil{K}^{\prime, c}_S=\wtil{\bm{K}}^c_{n+1} \cap S_{n+1}(\roi_F)$ and \index{Ks@$K^{\prime, c}_S$} $K^{\prime, c}_S= \bm{K}^c_{n+1} \cap S_{n+1}(\roi_F)$.
To prove \Cref{hom_fl} we need the following transfer statement, which we will prove in \Cref{pf_inhom_fl}.

\begin{lem}[inhomogeneous version]
\label{inhom_fl}
With notations as in the paragraph above, the pair $(\id_{\wtil{K}_{n+1}^c},0) \in C^\infty_c(U(V_0)(F)) \times C^\infty_c(U(V_1)(F))$ matches $ \vol(\wtil{K}_n)\vol(\GL_n(\roi_F))^{-1} \id_{\wtil{K}^{\prime, c}_S} \in C^\infty_c(S_{n+1}(F))$.
\end{lem}

\begin{proof}[Proof of \Cref{hom_fl} assuming \Cref{inhom_fl}] 
Assume \Cref{inhom_fl} holds, the pair $(\id_{K^c_{n+1}},0) \in C^\infty_c(U(V_0)(F)) \times C^\infty_c(U(V_1)(F))$ matches $\vol(U(W_0)(\roi_F))\vol(\GL_n(\roi_F))^{-1}\id_{K^{\prime, c}_S} \in C^\infty_c(S_{n+1}(F))$.
Indeed, for any $x=(x_{ij})\in U(V_0)(F)^\rs$, let $y=(y_{ij}) \in S_{n+1}(F)^\rs$ be an element that matches $x$, then  there exists $g\in GL_{n}(E)$ such that $y=g^{-1}xg$. Note that the $(n+1, n+1)$-th entry of $x$ does not change under the action of $GL_{n}(E)$.
So when $x_{n+1, n+1} \equiv 1 \md{\vp^c}$ and $h\in GL_n(E)$, $hxh^{-1}\in \bm{K}^c_{n+1}$ if and only if $hxh^{-1}\in \wtil{\bm{K}}^c_{n+1}$, and similarly for $y$. Hence by \Cref{inhom_fl}
\begin{align*}
\Orb(x,\id_{K^c_{n+1}})&=  \Orb(x, \id_{\wtil{K}^c_{n+1}})\\
&=\vol(\wtil{K}_n)\vol(\GL_n(\roi_F))^{-1} \Omega_S(y)\Orb(y, \id_{\wtil{K}^{\prime, c}_S}) \\
&= \vol(\wtil{K}_n)\vol(\GL_n(\roi_F))^{-1}\Omega_S(y)\Orb(y, \id_{K^{\prime, c}_S}),
\end{align*}
when $x_{n+1, n+1} \equiv 1 \md{\vp^c}$, and 
\[\Orb(x,\id_{K^c_{n+1}})=0=\vol(\wtil{K}_n)\vol(\GL_n(\roi_F))^{-1}\Omega_S(y)\Orb(y, \id_{K^{\prime, c}_S})\]
when $x_{n+1, n+1} \not \equiv 1 \md{\vp^c}$.

Now for matching $x=(x_{ij})\in U(V_1)(F)^\rs$ and $y=(y_{ij})\in S_{n+1}(F)^\rs$, 
\[\Orb(y, \id_{K^{\prime, c}_S})=\begin{cases}
0 & \text{ if }y_{n+1, n+1} \not \equiv 1 \md{\vp^c},\\
\Orb(y, \id_{\wtil{K}^{\prime, c}_S}) = \Orb(x,0) = 0 & \text{ if } y_{n+1, n+1} \equiv 1 \md{\vp^c}.
\end{cases}\]
Hence $(\id_{K^c_{n+1}},0)$ matches $\vol(U(W_0)(\roi_F))\vol(\GL_n(\roi_F))^{-1}\id_{K^{\prime, c}_S}$.

It is easy to see $\tilde{\id}_{K^c}=\vol(K_n)\id_{K^c_{n+1}}$ (see \Cref{f_tilde} for the definition of $\tilde{\id}_{K^c}$). 
To see that $\tilde{\id}_{K^{\prime, c}}$ (as defined in \Cref{f_tilde_s}) is a multiple of $\id_{K^{\prime, c}_{S}}$ takes a little more work. For any $y\in S_{n+1}(F)^\rs$, choose $g\in GL_{n+1}(E)$ such that $r(g)=y$, then
\begin{align*}
\tilde{\id}_{K^{\prime, c}}(y)&=\int_{GL_n(E)}\int_{GL_{n+1}(F)} \id_{K'_n \times K'_{n+1}}(h_1,h_1gh_2)(-1)^{\nu(gh_2)} dh_2 dh_1\\
&=\int_{K'_n}\int_{GL_{n+1}(F)} \id_{K^{\prime, c}_{n+1}}(h_1gh_2)(-1)^{\nu(gh_2)} dh_2 dh_1\\
&=\int_{K'_n}\int_{GL_{n+1}(F)} \id_{K^{\prime, c}_{n+1}}(gh_2) (-1)^{\nu(gh_2)} dh_2 dh_1\\
&=\vol(K'_n)\int_{GL_{n+1}(F)} \id_{K^{\prime, c}_{n+1}}(gh_2) dh_2,
\end{align*}
where the third step is because $K'_n\subset K^{\prime, c}_{n+1}$ and the last step is because the determinant of any element in $K^{\prime, c}_{n+1}$ has valuation 0.
Note that if $gh_2\in K^{\prime, c}_{n+1}$ for some $h_2\in GL_{n+1}(F)$, then $gh_2'\in K^{\prime, c}_{n+1}$ for $h_2'\in GL_{n+1}(F)$ if and only if $h_2'\in h_2(\GL_{n+1}(F)\cap K^{\prime, c}_{n+1})$. 
\Cref{inhom_to_hom} below implies that if $y\notin K^{\prime, c}_S$ then $\tilde{\id}_{K^{\prime, c}}(y)=0$ and if $y\in K^{\prime, c}_S$ then $\tilde{\id}_{K^{\prime, c}}(y)=\vol(K'_n)\vol(\GL_{n+1}(F)\cap K^{\prime, c}_{n+1})$. 
In other words, 
\[\tilde{\id}_{K^{\prime, c}}=\vol(\GL_{n+1}(F)\cap K^{\prime, c}_{n+1})\vol(K'_n)\id_{K^{\prime, c}_S}.\]

For matching $(x_1, x_2)\in G_1(F)^\rs$ and $(y_1, y_2)\in (G')(F)^\rs$, 
\[\Orb((y_1, y_2), \id_{K^{\prime, c}})= \Orb(y_1^{-1}y_2, \vol(\GL_{n+1}(F)\cap K^{\prime, c}_{n+1})\vol(K'_n)\id_{K^{\prime, c}_S})=\Orb(x_1^{-1}x_2,0)=0.\]

For any matching $(x_1, x_2)\in G_0(F)^\rs$ and $(y_1, y_2)\in (G')(F)^\rs$, Equations (\ref{u_hom_to_inhom}), (\ref{s_hom_to_inhom}) and (\ref{vol_GL_KF}) imply that 
\begin{align*}
\Orb((x_1, x_2),\id_{K^c})
&=\Orb(x_1^{-1}x_2, \vol(K_n)\id_{K^c_{n+1}}) \\
&= \vol(K_n)^2 \vol(\GL_n(\roi_F))^{-1}\Omega_S(\nu(y_1^{-1}y_2)) \Orb(\nu(y_1^{-1}y_2), \id_{K^{\prime, c}_S})\\ 
&= \frac{\vol(K_n)^2 \Omega_{G'}((y_1, y_2)) \Orb((y_1, y_2),\id_{K^{\prime, c}})}{\vol(\GL_n(\roi_F)) \vol(K'_n) \vol(\GL_{n+1}(F)\cap K^{\prime, c}_{n+1})} .
\end{align*}
Now the theorem follows from Equations (\ref{vol_GL_OF}), (\ref{vol_GL_OE}) and (\ref{vol_UW_OF}).
\end{proof}

\begin{lem}
\label{inhom_to_hom}
For any $g\in GL_{n+1}(E)$, $r(g)\in K^{\prime, c}_S$ if and only if $g\GL_{n+1}(F)\cap K^{\prime, c}_{n+1} \neq 0$.
\end{lem}

\begin{proof}
Recall that $\wtil{K}^{\prime, c}_{n+1}=\wtil{\bm{K}}^c_{n+1} \cap GL_{n+1}(\roi_E)$ and $\wtil{K}^{\prime, c}_S=\wtil{\bm{K}}^c_{n+1} \cap S_{n+1}(\roi_F)$. We first show that $r(g)\in \wtil{K}^{\prime, c}_S$ implies $gGL_{n+1}(F)\cap \wtil{K}^{\prime, c}_{n+1} \neq 0$. Let $\Lambda_1=\roi_E^{n+1} \subset \Lambda_2 = \roi_E^n + \vp^{-c}\roi_E$ be lattices in $E^{n+1}$, then $\wtil{K}^{\prime, c}_{n+1}=\Stab_{GL_{n+1}(E)}(\Lambda_1) \cap \Stab_{GL_{n+1}(E)}(\Lambda_2)$. 
Suppose that $r(g)=g\bar{g}^{-1}\in \wtil{K}^{\prime, c}_S \subset \wtil{K}^{\prime, c}_{n+1}$, then for $i=1,2$ we have $g\bar{g}^{-1}\Lambda_i=\Lambda_i$, i.e., $g^{-1}\Lambda_i=\bar{g}^{-1}\Lambda_i=\br{g^{-1}\Lambda_i}$ is Galois-invariant. 
Let $L_i=g^{-1}\Lambda_i\cap F^{n+1}$, then the structure theorem over principal ideal domains implies that they are free of rank $n+1$ and there exists an $\roi_F$-basis $f_1, \cdots, f_{n+1}$ of $L_2$ such that $\vp^{t_1}f_1, \cdots, \vp^{t_{n+1}}f_{n+1}$ is an $\roi_F$-basis of $L_1$ for some $t_1\geq \cdots \geq t_{n+1}\geq 0$. Then $g^{-1}\Lambda_1=\vp^{t_1}f_1\roi_E+ \cdots + \vp^{t_{n+1}}f_{n+1}\roi_E$ and $g^{-1}\Lambda_2=f_1\roi_E + \cdots + f_{n+1}\roi_E$. 
Note that 
\[\roi_E/\vp^c \roi_E \cong \Lambda_2/\Lambda_1 \cong  g^{-1}\Lambda_2/g^{-1}\Lambda_1 = \oplus_i (\roi_E/\vp^{t_i}\roi_E),\]
so $t_1=c$ and $t_2=\cdots=t_{n+1}=0$. Replace $f_{n+1}$ by $\vp^c f_{n+1}$ and let $h=(f_1 \ \cdots \ f_{n+1}) \in GL_{n+1}(F)$ be the matrix whose columns are $f_i$ (with respect to the standard basis), we get $g^{-1}\Lambda_i=h \Lambda_i$ for $i=1,2$. Hence $gh \in \Stab_{GL_{n+1}(E)}(\Lambda_1) \cap \Stab_{GL_{n+1}(E)}(\Lambda_2) = \wtil{K}^{\prime, c}_{n+1}$.
 
If $r(g)\in K^{\prime, c}_S$, by the argument above there exists $h\in GL_{n+1}(F)$ such that $g_0=gh\in \wtil{K}^{\prime, c}_{n+1}$. 
Suppose that $g_0=\mat{A&b\\c&d}$ and $g_0^{-1}=\mat{A'&b'\\c'&d'}$, where $A, A'\in \Mat_{n}(E),\ b, b'\in E^n,\ c, c'\in E_n$ and $d, d'\in E$. Then $cb'+dd'=1$ and $r(g_0)=g_0 \bar{g}_0^{-1}\in K^{\prime, c}_S$, so $c\bar{b}'+d\bar{d}'\equiv 1 \mod \vp^c$. 
Since $b'\equiv 0 \mod \vp^c$, we have $dd' \equiv d\bar{d}' \equiv 1 \mod \vp^c$.
Suppose that $d'=d_1+\alpha d_2$ with $d_1, d_2\in \roi_F$ where $\alpha\in \roi_E^\times$ is such that $\alpha^2\in F$. The congruences above implies $d\in \roi_E^\times$ and $\vp^c| 2\alpha d d_2$, hence $\vp^c|d_2$ and $1\equiv dd_1 \mod \vp^c$.
Let $h'=\mat{I_n&\\&d_1}$, then $h'\in GL_{n+1}(F)$ and $g_0h' \in K^{\prime, c}_{n+1}$, i.e., $ghh'\in K^{\prime, c}_{n+1}$.

If $g\GL_{n+1}(F)\cap K^{\prime, c}_{n+1} \neq 0$, choose $h\in \GL_{n+1}(F)$ such that $g'=gh\in K^{\prime, c}_{n+1}$, then $r(g)=g\br{g}^{-1} = (gh)(\br{h}^{-1}\br{g}^{-1}) =g'\br{g'}^{-1} \in K^{\prime, c}_{n+1}\cap S_{n+1}(F)=K^{\prime, c}_S$.
\end{proof}

\subsection{Proof of \Cref{inhom_fl}}
\label{pf_inhom_fl}
We convert \Cref{inhom_fl} into the Lie algebra transfer statement \Cref{lie_alg_fl}, which we will prove using Jacquet--Rallis's fundamental lemma. 
Let \index{kct@$\tilde{\lk}_c$} \index{kctp@$\tilde{\lk}'_c$} $\tilde{\lk}_c=\wtil{\bm{K}}^c_{n+1} \cap \ul(V_0)(F)$ and $\tilde{\lk}'_c= \wtil{\bm{K}}^c_{n+1} \cap \s_{n+1}(F)$. 

\begin{lem}[Lie algebra version]
\label{lie_alg_fl}
With notations above, the pair $(\id_{\tilde{\lk}_c},0) \in C^\infty_c(\ul(V_0)(F)) \times C^\infty_c(\ul(V_1)(F))$ matches $\vol(K_n)\vol(\GL_n(\roi_F))^{-1} \id_{\tilde{\lk}'_c} \in C^\infty_c(\s_{n+1}(F))$.
\end{lem}

To prove it, we first state two technical results.
\begin{lem}
\label{caylay}
If $g\in \wtil{\bm{K}}^c_{n+1}, \xi\in E$, $\xi\bar{\xi}=1$ and $\det(g+\xi) \not\equiv 0 \md{\vp}$, then $\frac{g-\xi}{g+\xi} \in \wtil{\bm{K}}^c_{n+1}$.
\end{lem}
  
\begin{proof}
Note that ${\wtil{\bm{K}}^c_{n+1} = \Mat_{n+1}(\roi_E) \cap \mat{\roi_E & \pp_E^c \\ \pp_E^{-c} & \roi_E}}$ and $g, g+\xi , g-\xi \in \wtil{\bm{K}}^c_{n+1}$. 
  
Let $\Lambda_1=\roi_E^{n+1}$ and $\Lambda_2=\roi_E^n+\pp_E^{-c}$, then $(g+\xi)\Lambda_i\subset \Lambda_i$ for $i=1,2$ and $\det (g+\xi)\in \roi_E^\times$ implies that $g+\xi \in \Stab(\Lambda_1)\cap \Stab(\Lambda_2)$ and hence $(g+\xi)^{-1} \in \Stab(\Lambda_1)\cap \Stab(\Lambda_2)$. Thus $\frac{g-\xi}{g+\xi}\Lambda_i \subset \Lambda_i$ for $i=1,2$, i.e., $\frac{g-\xi}{g+\xi} \in \wtil{\bm{K}}^c_{n+1}$.
\end{proof}
  
\begin{lem}
\label{trans_factor}
Suppose that $X=\mat{A & b\\ c& d}$ with $A\in \Mat_m(E), b\in E^m, c\in E_m, d\in E$ and \index{e0@$e_0$} $e_0=(0, \cdots, 0,1)^T \in E^{m+1}$, we have
  \[\det((X^ie_0)_{0\leq i\leq m})=(-1)^m \det((A^ib)_{0\leq i\leq m-1}).\]
\end{lem}
  
\begin{proof}
Note that for $i>0$, straightforward induction shows that $X^ie_0=\mat{A^{i-1}b+b_i\\ d_i}$, where $b_i$ is a linear combination of $b, Ab, \cdots, A^{i-2}b$ and $d_i\in E$. Then 
\begin{align*}
\det((X^ie_0)_{0\leq i\leq m}) 
&= \det{\mat{
0 & b & Ab+b_1 & \cdots & A^{m-1}b+b_{m-1} \\
1 & d & d_1 & \cdots & d_{m-1}
}}\\ 
&= (-1)^m \det \mat{b & Ab+b_1 & \cdots & A^{m-1}b+b_{m-1}}\\
&= (-1)^m \det((A^ib)_{0\leq i\leq m-1}),
\end{align*}
where the last step is because every $b_i$ is a linear combination of previous columns.
\end{proof}
  
\begin{proof}[Proof of \Cref{lie_alg_fl}]
We reduce it to \Cref{JR_fl}.
For $i=0,1$, let $V'_i=W_i\oplus^\perp \pa{u_i'}$ with $\pa{u_i',u_i'}=1$.

We define a map $\iota: \ul(V_i) \to \ul(V'_i)$ by
\[\iota:\begin{blockarray}{ccc}
 & n &  1  \\
\begin{block}{c(cc)}
  n & x & y \\
  1 & z & w \\
\end{block}
\end{blockarray} \mapsto \mat{x & \vp^{-c}y \\ z & w}.\]
When $i=0$ this map restricts to a bijection $\tilde{\lk}_c\to \tilde{\lk}_0:=\ul(V'_0)(\roi_F)$.
Theorem 6.1 of \cite{Rallis-Schiffmann} states that $X=\mat{x&y\\z&w} \in \Mat_{n+1}(E)$ is regular semisimple if and only if $\{x, xy, \cdots, x^{n-1}y\}$ is a basis of $E^n$ and $\{z, zx, \cdots, zx^{n-1}\}$ is a basis of $E_n$.
This theorem implies that $X\in \ul(V_i)$ is regular semisimple if and only if $\iota(X)$ is regular semisimple. 
Moreover, for $X=\mat{x&y\\z&w}\in \ul(V_0)(F)^\rs$ we have
\begin{align*}
\Orb(X, \id_{\tilde{\lk}_c})
&= \int_{U(W_0)(F)} \id_{\tilde{\lk}_c}(h^{-1}Xh) dh\\
&= \int_{U(W_0)(F)} \id_{\tilde{\lk}_c}(\mat{h^{-1}xh & h^{-1}y \\ zh & w}) dh\\
&= \int_{U(W_0)(F)} \id_{\tilde{\lk}_0}(\iota(\mat{h^{-1}xh & h^{-1}y \\ zh & w})) dh\\
&= \int_{U(W_0)(F)} \id_{\tilde{\lk}_0}(\mat{h^{-1}xh & \vp^{-c}h^{-1}y \\ zh & w}) dh\\
&= \int_{U(W_0)(F)} \id_{\tilde{\lk}_0}(h^{-1}\iota(X)h) dh\\
&= \Orb(\iota(X), \id_{\tilde{\lk}_0}),
\end{align*}
where $\Orb(\iota(X), \id_{\tilde{\lk}_0})$ is the orbital integral on $\ul(V'_0)$ under the conjugation action by $U(W_0)$.

By abuse of notation, we denote the map 
\begin{align*}
\s_{n+1}(F) &\to \s_{n+1}(F)\\
\begin{blockarray}{ccc}
 & n &  1  \\
\begin{block}{c(cc)}
  n & x & y \\
  1 & z &  w\\
\end{block}
\end{blockarray} &\mapsto \mat{x & \vp^{-c}y \\ z & w}
\end{align*}
also by $\iota$. Again it is obvious that $\iota: \tilde{\lk}'_c \to \s_{n+1}(\roi_F)$ is a bijection and for $Y\in \s_{n+1}(\roi_F)^\rs$ we have
\[\Orb(Y,\id_{\tilde{\lk}_c})= \Orb(\iota(Y), \id_{\s_{n+1}(\roi_F)})\]
for the same reason as above. It is straightforward to check that $X\in \ul(V_i)(F)^\rs$ matches $Y\in \s_{n+1}(F)^\rs$ if and only if $\iota(X) \in \ul(V'_i)(F)^\rs$ matches $\iota(Y)\in \s_{n+1}(F)^\rs$. To summarize, the diagram 
\[\begin{matrix}
\s_{n+1}(F)^\rs/\GL_n(F) & = & \ul(V_0)(F)/U(W_0)(F) & \sqcup & \ul(V_1)(F)/U(W_1)(F)\\
\downarrow \iota & & \downarrow \iota  & & \downarrow \iota \\
\s_{n+1}(F)^\rs/\GL_n(F) & = & \ul(V'_0)(F)/U(W_0)(F) & \sqcup & \ul(V'_1)(F)/U(W_1)(F)
\end{matrix}\]
commutes under the matching condition.
Hence \Cref{JR_fl} implies that 
\begin{align*}
\omega(\iota(Y)) \Orb(Y,\id_{\tilde{\lk}_c})
&= \omega(\iota(Y))\Orb(\iota(Y), \id_{\s_{n+1}(\roi_F)})\\
&=\begin{cases}
  \vol(U(W_0)(\roi_F))^{-1}\vol(\GL_n(\roi_F)) \Orb(\iota(X), \id_{\tilde{\lk}_0}) & \text{ if }i=0,\\
  \Orb(\iota(X),0) & \text{ if }i=1
\end{cases}\\
&=\begin{cases}
    \vol(U(W_0)(\roi_F))^{-1}\vol(\GL_n(\roi_F)) \Orb(X, \id_{\tilde{\lk}_c})& \text{ if }i=0,\\
    \Orb(X,0) & \text{ if }i=1.
\end{cases}
\end{align*}
for matching $X\in \ul(V_i)(F)^\rs$ and $Y\in \s_{n+1}(F)^\rs$.
Note that for $Y=\mat{x & y\\z& w}$ we have $\iota(Y)=\mat{x & \vp^{-c}y\\z& w}$ and \Cref{trans_factor} implies that
\begin{align*}
\omega(\iota(Y))
&=\tilde{\eta}_{E/F}(\det(e_0\ \iota(Y)^Te_0\ \cdots \ (\iota(Y)^T)^ne_0))\\
&= \tilde{\eta}_{E/F}(\det(z\ x^Tz\ \cdots\ (x^T)^{n-1}z))
=\tilde{\eta}_{E/F} (\det(e_0\ Y^Te_0\ \cdots \ (Y^T)^ne_0))=\omega(Y).
\end{align*}

Hence $(\id_{\tilde{\lk}_c},0)$ matches $\vol(U(W_0)(\roi_F))\vol(\GL_n(\roi_F))^{-1} \id_{\tilde{\lk}'_c}$ with transfer factor $\omega$ as in the Jacquet-Rallis case. 
\end{proof}

\begin{proof}[proof of \Cref{inhom_fl}]
For $i=0,1$ and $g_i\in U(V_i)(F)^\rs$, choose $\xi\in E^\times$ such that $\xi\br{\xi}=1$ and $\det(g_i+\xi)\not\equiv 0 \mod \vp$. Note that there are $q_F+1$ elements in $\roi_E/\vp\roi_E$ with norm 1 and $\det(g_i+\xi)\not\equiv 0\mod \vp$ if and only if $-\xi$ is not an eigenvalue of $g_i$ module $\vp$.
Since we have assumed $q_F+1> n+1$ and $g_i$ has at most $n+1$ distinct eigenvalues, such $\xi$ has to exist.
For any $h_i\in U(W_i)(F)$, $\cay_\xi^{-1}(h_ig_ih_i^{-1})$ is in $\ul(V_i)(F)$ (see \Cref{section_notations}) and $\det(h_ig_ih_i^{-1}+\xi)\not\equiv 0 \mod \vp$.

Next we show that for $g\in U(V_0)(F)^\rs$ and $h\in U(W_0)(F)$, $hgh^{-1}\in \wtil{K}^c_{n+1}$ if and only if $\cay_\xi^{-1}(hgh^{-1})\in \tilde{\lk}_c$ .
If $g'=hgh^{-1}\in \wtil{K}^c_{n+1}$, then it is in $\wtil{\bm{K}}^c_{n+1}$ and $\cay_\xi^{-1}(g')= \frac{g'-\xi}{g'+\xi}\in \wtil{\bm{K}}^c_{n+1}$ by \Cref{caylay}, and thus $\cay_\xi^{-1}(g')\in \wtil{\bm{K}}^c_{n+1} \cap \ul(V_0)(F) = \tilde{\lk}_c$. 
On the other hand, if $X'=\cay_\xi^{-1}(g')\in \tilde{\lk}_c$, then $X'\in \wtil{\bm{K}}^c_{n+1}$ and the determinant of $X'-1=-\frac{2\xi}{g'+\xi}$ is in $\roi_E^\times$. 
Again by \Cref{caylay} we have $\frac{1+X'}{X'-1}\in \wtil{\bm{K}}^c_{n+1}$ and $g'=-\xi\frac{X'+1}{X'-1}\in \wtil{\bm{K}}^c_{n+1}\cap U(V_0)(F)=\wtil{K}^c_{n+1}$. 
Hence
\begin{align*}
\Orb(g,\id_{\wtil{K}^c_{n+1}}) &= \int_{U(W_0)(F)} \id_{\wtil{K}_{n+1}}(hgh^{-1}) dh\\
&=\int_{U(W_0)(F)} \id_{\tilde{\lk}_c}(\cay_\xi^{-1}(hgh^{-1})) dh\\
&=\int_{U(W_0)(F)} \id_{\tilde{\lk}_c}(h\cay_\xi^{-1}(g)h) dh\\
&=\Orb(\cay_\xi^{-1}(g), \id_{\tilde{\lk}_c}).
\end{align*}

It follows from definitions that $g_i \in U(V_i)(F)^\rs$ matches $s \in S_{n+1}(F)^\rs$ if and only if $\cay_\xi^{-1}(g_i)$ matches $\cay_\xi^{-1}(s)$.
Let $s\in S_{n+1}(F)^\rs$ be an element that matches $g_i \in U(V_i)(F)^\rs$ for $i=0$ or 1, then $\det(s+\xi)=\det(g_i+\xi)$ and for the same reason as above we have
\[\Orb(s,\id_{\wtil{K}^{\prime, c}_{n+1}})= \Orb(\cay_\xi^{-1}(s), \id_{\tilde{\lk}'_c}).\]
Since $\omega(\cay_\xi^{-1}(s))= \Omega(s)$ (see, for example, \cite{RSZ2} Lemma 11.9), \Cref{lie_alg_fl} implies that 
\begin{align*}
\Orb(g_0,\id_{\wtil{K}^c_{n+1}}) 
&= \Orb(\cay_\xi^{-1}(g_0), \id_{\tilde{\lk}_c}) \\
&=\vol(U(W_0)(\roi_F))\vol(\GL_n(\roi_F))^{-1} \omega(\cay_\xi^{-1}(s))\Orb(\cay_\xi^{-1}(s), \id_{\tilde{\lk}'_c}) \\
&= \vol(U(W_0)(\roi_F))\vol(\GL_n(\roi_F))^{-1} \Omega(s)\Orb(s,\id_{\wtil{K}^{\prime, c}_{n+1}})
\end{align*}
when $i=0$, and 
\[\Orb(g_1,0)=0=\Orb(\cay_\xi^{-1}(g_1),0)= \Orb(\cay_\xi^{-1}(s), \id_{\tilde{\lk}_c}) = \Orb(s, \id_{\wtil{K}'_{n+1}})\]
when $i=1$. Hence the lemma is proved.
\end{proof}

\section{Calculation of local characters}
\label{section_calc_I}
Notations in \Cref{base_change} and \Cref{section_whittaker} are independent of the conventions in previous sections. 

\subsection{Local Langlands correspondence and base change}
\label{base_change}
Let $E/F$ be an unramified quadratic extension of $p$-adic local fields.
For a connected reductive group $G$ over $E$, let $\Pi_\temp(G)$ denote the set of equivalence classes (under conjugation by $\widehat{G}$) of tempered irreducible representations of $G$ and let $\Phi_\temp(G)$ be the set of equivalence classes of tempered $L$-parameters of $G$. 
Let $W(E)$ denote the Weil group of $E$ and let $WD(E)=W(E) \times SL_2(\C)$ denote the Weil-Deligne group of $E$, and similarly define $W(F)$ and $WD(F)$. 
For an $L$-parameter $\varphi: WD(E) \to {}^L G=\widehat{G} \times \Gal(E)$ of $G(E)$, let $C_\varphi$ be the centralizer of the image of $\varphi$ in $\widehat{G}$ and denote by $A_\varphi=\pi_0(C_\varphi)$ the component group of $C_\varphi$.
Representations of $WD(F)$ and $WD(E)$ always refer to admissible ones.

The work of Mok \cite[Theorem 3.2.1]{Mok} and Kaletha, Minguez, Shin, and White \cite[Theorem 1.6.1]{KMSW} established the local Langlands correspondence for unitary groups, a part of which we state in \Cref{LLC} below.
Let $V^+$ (resp. $V^-$) be an $n$-dimensional hermitian space over $E$ with discriminant 1 (resp. $-1$), then the $L$-groups of $U(V^+)$ and $U(V^-)$ are isomorphic. 
We denote $\Phi_\temp(U(V^+))=\Phi_\temp(U(V^-))$ by $\Phi_\temp(U_n)$. 

\begin{thm}
\label{LLC}
To every tempered $L$-parameter $\varphi \in \Phi_\temp(U_n)$ we can associate canonically a finite set $\Pi_\varphi \subset \Pi_\temp(U(V^+)) \sqcup \Pi_\temp(U(V^-))$ such that 
\[\Pi_\temp(U(V^+)) \sqcup \Pi_\temp(U(V^-))=\underset{\varphi\in \Phi_\temp(U_n)}{\sqcup} \Pi_\varphi\]
and these sets are classified by endoscopic relations. The $\Pi_\varphi$'s are called Vogan $L$-packets.

Moreover, for every $\varphi \in \Phi_\temp(U_n)$ there is a bijection
\begin{equation*}
J_\psi: \Pi_\varphi \to \hom(A_\varphi,\C^\times).
\end{equation*}
that depends on the choice of a $\text{Nm}_{E/F}(E^\times)$-orbit of non-trivial additive characters $\psi: E/F \to \C^\times$.
\end{thm}
 
Fix $s\in W(F)\setminus W(E)$. Let $M$ be a representation of $WD(E)$. We say that $M$ is conjugate-dual of sign 1 or conjugate-orthogonal if there exists a nondegenerate bilinear form $B: M \times M \to \C$ such that 
\[B(\tau m_1, s \tau s^{-1} m_2)=B(m_1, m_2) \quad \text{and} \quad B(m_2, m_1)=B(m_1, s^2 m_2) \quad \text{for all }\tau \in WD(E).\]
We say that $M$ is conjugate-dual of sign $-1$ or conjugate-symplectic if there exists a nondegenerate bilinear form $B: M \times M \to \C$ such that 
\[B(\tau m_1, s \tau s^{-1} m_2)=B(m_1, m_2) \quad \text{and} \quad B(m_2, m_1)=-B(m_1, s^2 m_2) \quad \text{for all }\tau \in WD(E).\]
The notions of conjugate-orthogonal and conjugate-symplectic are independent of the choice of $s$.

Every $L$-parameter $\varphi \in \Phi_\temp(U_n)$ can be regarded as a sign $(-1)^{n+1}$ conjugate-dual representation $M$ of $WD(E)$ (see \cite[Section 8]{GGP}), and below we will regard $L$-parameters of $U(V^\pm)$ as conjugate-dual representations of $WD(E)$. 
The local Langlands correspondence of general linear groups associates a representation $\Pi(\varphi)$ of $\GL_n(E)$ to $\varphi$. 
For any $\pi$ in the Vogan $L$-packet $\Pi_\varphi$, we call $\Pi(\varphi)$ the base change of $\pi$ and write \index{BC@$BC$} $BC(\pi):=\Pi(\varphi)$.  

For an irreducible admissible representation $\Pi$ of $\GL_n(E)$ and a non-trivial additive character $\psi$ of $E$ with conductor $\roi_E$, Godement and Jacquet \cite{G-J} defined the local factor $\epsilon(s,\Pi,\psi)$, which is of the form 
\[\epsilon(s,\Pi,\psi)=\epsilon_\Pi(q_E^{-s},\psi),\]
where $\epsilon_\Pi(X,\psi)=c_0X^m$ for some $c_0\in \C$ and $m\geq 0$. The integer $m$ does not depend on the choice of $\psi$ and is called the conductor of $\Pi$.
For a representation $\pi\in \Pi_\varphi$ with $\varphi \in \Phi_\temp(U_n)$, we define the conductor of $\pi$ to be the conductor of $BC(\pi)$. Sometimes we will call it the conductor of $\varphi$ or $\Pi_\varphi$ since it does not depend on the choice of $\pi \in \Pi_\varphi$.

\subsubsection*{Asai L-function}
Let $M$ be an $m$-dimensional vector space over $\C$, Gan, Gross, and Prasad defined and characterized in \cite[Chapter 7]{GGP} two representations $As^\pm(M)$ of $H^1={}^L \Res_{E/F}\GL_m= \GL(M) \times \GL(M) \rtimes W_F$ as follows. Let Std denote the standard representation of $\GL(M)$ on $M$ and let $H^0= \GL(M) \times \GL(M) \rtimes W_E$. We view the representation $\text{Std} \otimes \text{Std}$ of $\GL(M)\times \GL(M)$ on $M \otimes M$ as a representation of $H^0$ by letting $W_E$ act trivially. Then \index{Asai@$As^+, As^-$}
\[\Ind_{H^0}^{H^1}\text{Std} \otimes \text{Std}= As^+(M) \oplus As^-(M)\]
for some representations $As^+(M), As^-(M)$ satisfying
\[\tr(w|As^\pm)=\pm m\]
for $w\in W_F \setminus W_E$. 
The representations $As^\pm$ satisfy the relation 
\[L(s,\pi, As^-)=L(s,\pi\otimes \tilde{\eta}_{E/F}, As^+)\]
for representations $\pi$ of $\GL_m(E)$. On the other hand, Flicker \cite{Flicker} defined the Asai $L$-function of representations of $\GL_n(E)$ using Rankin--Selberg integrals and showed that under this definition, the Asai $L$-function of an unramified generic unitary representation $\pi$ of $\GL_n(E)$ with Satake parameters $\alpha_1, \cdots, \alpha_n$ is 
\[L_{As}(s,\pi)=\prod_{1\leq i\leq n}(1-q_F^{-s}\alpha_i)^{-1} \prod_{1\leq j<k\leq n} (1-q_F^{-2s}\alpha_j\alpha_k)^{-1}.\]

Matringe \cite[Theorem 4.3]{matringe_asai} proved that for generic representations $\pi$ of $\GL_n(E)$ the two definitions coincide:
\[L_{As}(s,\pi)=L(s,\pi, As^+).\]
Consequently, if $\pi$ is unramified with Satake parameters $\alpha_1, \cdots, \alpha_n$, then 
\[L(s,As^-)=\prod_{1\leq i\leq n}(1+q_F^{-s}\alpha_i)^{-1} \prod_{1\leq j<k\leq n} (1-q_F^{-2s}\alpha_j\alpha_k)^{-1}.\]

\subsection{Essential Whittaker function}
\label{section_whittaker}
Let $E$ be a nonarchimedean local field with residue field $\F_{q_E}$.
Let $G_n=\GL_n(E)$ and let $\nu$ be the valuation map on $E^\times$ defined in \Cref{section_notations}. For $m=tl$ and a cuspidal representation $\rho$ of $G_l$, the representation \index{1segment@$[\nu^{-(t-1)}\rho, \rho]$} 
$[\nu^{-(t-1)}\rho,\rho]:=\nu^{-(t-1)}\rho \times \cdots \times \nu^{-1}\rho \times \rho$ 
of $M=G_l \times \cdots \times G_l$ can be regarded as a representation of the standard parabolic group $P \subset G_m$ with Levi component $M$. The normalized (i.e. twisted by $\delta_p^{1/2}$) induced representation 
$\Ind_{P}^{G_m}[\nu^{-(t-1)}\rho,\rho]$
of $G_m$ has a unique irreducible quotient, which we again denote by $[\nu^{-(t-1)}\rho, \rho]$. Representations of this form are called segments and sometimes denoted by \index{Delta@$\Delta$} $\Delta$.
Recall that all generic admissible representations of $G_m$ are of the form $\Ind_{P}^{G_m} (\Delta_1 \times \cdots \times \Delta_r)$, 
where $P$ is the standard parabolic subgroup of $G_m$ with Levi component $G_{n_1} \times \cdots \times G_{n_r}$ for some $n_1, \cdots, n_r$ that sum up to $m$ and $\Delta_1, \cdots, \Delta_r$ segments of $\GL_{n_1}, \cdots, \GL_{n_r}$ respectively \cite{Zelevinsky}. 
For simplicity, we denote this representation of $G_m$ by $\Delta_1 \times \cdots \times \Delta_r$.
In particular, unramified representations are of the form $\Ind_{B_m}^{G_m} \chi$, where \index{Bm@$B_m$} $\chi=(\chi_1,\cdots, \chi_m)$ and $\chi_i$ are unramified characters of $E^\times$ and $B_m\subset G_m$ is the Borel subgroup of upper triangular matrices.

For an unramified representation $\Pi=\Ind_{B_m}^{G_m} \chi$ of $G_m$, the spherical Whittaker function is of the following form (see, for example, \cite{B-G_spherical-whit-fcn}).
Let \index{alpha@$\alpha, \alpha_i$} $\alpha_i=\chi_i(\vp)$, and for $\lambda=(\lambda_1,\cdots, \lambda_m) \in \C^m$ write \index{pilambda@$\vp^\lambda$} $\vp^\lambda=\diag(\vp^{\lambda_1}, \cdots, \vp^{\lambda_m})$, then the normalized spherical Whittaker function $W^0_\Pi$ is defined by
\begin{equation}
\label{sph_whit}
W_\Pi^0(\vp^\lambda)=\begin{cases}
    \delta_m(\vp^\lambda)^{1/2} s_\lambda(\alpha)& \text{ if }\lambda_1 \geq \cdots \geq \lambda_m, \\
    0& \text{ otherwise},
\end{cases}
\end{equation}
where \index{slambda@$s_\lambda(\alpha)$}
\[s_\lambda(\alpha)=\frac{\det\ (\alpha_i^{\lambda_j+n-j})_{1\leq i,j\leq m}}{\det\ (\alpha_i^{n-j})_{1\leq i,j\leq m}}\]
and \index{\deltam@$\delta_m$} $\delta_m$ is the modular character of $B_m$.

Let $\psi:E \to \C^\times$ be an additive character with conductor $\roi_E$. Jacquet, Piatetski-Shapiro, and Shalika \cite{J-PS-S} proved the following result.
\begin{thm}
\label{essential_whit_fcn}
Let $\Pi_{n+1}$ be a generic representation of $GL_{n+1}(E)$ with conductor $c$ and Whittaker model $\W(\Pi_{n+1},\psi)$. 
There exists a unique element $W\in \W(\Pi_{n+1},\psi)$ such that
\begin{enumerate}[label=(\alph*)]
\item $W(gh)=W(g)$ for all $h\in GL_n(\roi_E)$, and
\item for any unramified representation $\Pi_n$ of $GL_n(E)$ with Satake parameter $\{q_E^{-a_1},\cdots, q_E^{-a_n}\}$, let $W^0\in \W(\Pi_n, \br{\psi})$ be the normalized spherical function (so that $W^0(1)=1$), then 
\[\lambda(W,W^0)=L(\frac{1}{2}, \Pi_{n} \times \Pi_{n+1})\]
with Rankin--Selberg $L$-function \cite{G-J} on the right hand side. 
\end{enumerate}
Moreover, $\Pi_{n+1}^{K^{\prime, c}_{n+1}}$ is a 1-dimensional space spanned by $W$, where $K^{\prime, c}_{n+1}$ is as defined in \Cref{section_measures}. We call the element $W^\ess\in \Pi_{n+1}^{K^{\prime, c}_{n+1}}$ with $W^\ess(1)=1$ the (normalized) essential Whittaker function and the elements in $\Pi_{n+1}^{K^{\prime, c}_{n+1}}$ are called newforms.
\end{thm}

\begin{rmk} 
Note that \cite[Theorem 5.1]{J-PS-S} is stated slightly differently from above. Let ${}^tK_{n+1}^{\prime, c}$ denote the group consisting of transposes of elements in $K_{n+1}^{\prime, c}$, \cite{J-PS-S} shows that $\Pi_{n+1}^{{}^t K^{\prime, c}_{n+1}}$ is one dimensional. However, the conductor of $\Pi_{n+1}$ is the same as the conductor of its contragredient representation $\hat{\Pi}_{n+1}$, hence $\dim \Pi_{n+1}^{K_{n+1}^{\prime, c}} = \dim \hat{\Pi}_{n+1}^{{}^t K_{n+1}^{\prime, c}}=1$. 
\end{rmk}

Matringe \cite{Matringe} proved that the essential Whittaker functions are of the following explicit form.

\begin{defn}
\label{pi_unram}
Let $\pi=\Delta_1 \times \cdots \times \Delta_t$ be a generic representation of $\GL_m(E)$, with $\Delta_i=\left[\nu^{-\left(k_i-1\right)} \rho_i, \rho_i\right]$. Let $r$ be the number of $\rho_i$'s that are unramified characters of $E^\times$. 
When $r>0$, we call these characters $\chi_1, \ldots, \chi_r$ and order them so that $\operatorname{Re}\left(\chi_i\right) \geq \operatorname{Re}\left(\chi_{i+1}\right)$ for $1 \leq i \leq r-1$, where \index{rechi@$\re(\chi)$} $\operatorname{Re}\left(\chi_i\right)$ means $\operatorname{Re}(t_i)$ for $\chi_i=|\cdot |^{t_i}$. 
We define $\pi_u$ to be the trivial representation of $G_0$ when $r=0$ and define it to be the unramified representation $\chi_1 \times \cdots \times \chi_r$ of $G_r$ when $r \geq 1$.
\end{defn}

\begin{thm}[\cite{Matringe}, Corollary 3.2]
\label{thm_ess_Whit}
For $m \geq 2$, let $\pi$ be a ramified generic representation of $\GL_m(E)$ and let $r$ be as in \Cref{pi_unram} above. Then $r<m$ and
\begin{align*}
&W^{\text{ess}}_\pi(\mat{\vp^{f_1}&&&\\ &\ddots&&\\ &&\vp^{f_{m-1}}&\\ &&&1})\\
&=\begin{cases}
W_{\pi_u}^0(\mat{\vp^{f_1}&&\\&\ddots&\\&&\vp^{f_r}})|\vp^{f_1+\cdots+f_r}|_E^{(m-r)/2} & \text{ if }f_1\geq \cdots \geq f_r\geq 0= f_{r+1}= \cdots = f_{m-1},\\
0 & \text{ otherwise}.
\end{cases}
\end{align*}
\end{thm}

\subsection{Proof of \Cref{thm:mainthm}}
\label{section_integral}
Let $\sigma_n$ and $\sigma_{n+1}$ be irreducible tempered unitary representations of $\GL_n(E)$ and $\GL_{n+1}(E)$ respectively, satisfying that $\sigma_n$ is unramified and $\sigma_{n+1}$ has conductor $c$. 
Let $\sigma=\sigma_n \boxtimes \sigma_{n+1}$ and let $\sigma_u$ be the unramified representation derived from $\sigma_{n+1}$ as in \Cref{pi_unram}.
\index{W@$\W_n, \W_{n+1}$}
Let $\psi:F\to \C^\times$ be a nontrivial additive character of $F$ with conductor $\roi_F$ and define $\psi_E: E\to \C^\times$ by $\psi_E(a)=\psi(\frac{1}{2} \tr_{E/F}(a))$.  
Write $\W_n= \W(\sigma_n, \br{\psi}_E)$ and $\W_{n+1} = \W(\sigma_{n+1},\psi_E)$ for simplicity. Note that $\dim \sigma_n^{K'_n}= \dim \sigma_{n+1}^{K^{\prime, c}_{n+1}} =1$.
\index{W@$W_n, W_{n+1}, W_n^0, W_{n+1}^0$}
Let $W_n\in \W_n$ (resp. $W_{n+1} \in \W_{n+1}$) be an element in $\sigma_n^{K'_n}$ (resp. $\sigma_{n+1}^{K^{\prime, c}_{n+1}}$) with norm 1 with respect to the inner product $\theta_n$ (resp. $\theta_{n+1}$). 
\index{c@$c_n, c_{n+1}$}
Let $W_i^0=c_i^{-1}W_i$ be the multiple of $W_i$ such that $W_i^0(1)=1$ for $i=n,n+1$ and let $\br{\sigma_n}$ (resp. $\br{\sigma_u}$) denote the unramified representation whose Satake parameters are complex conjugates of those of $\sigma_n$ (resp. $\sigma_u$).
Then by \cite[Proposition 2.3]{Jacquet-Shalika} we have 
\begin{align}
\label{eq_constants}
(c_n\br{c_n})^{-1}&=\theta_n(W_n^0, W_n^0)=\vol(\GL_{n-1}(\roi_E))L(1, \sigma_n \times \br{\sigma_n}),\\
(c_{n+1}\br{c_{n+1}})^{-1}&=\theta_{n+1}(W_{n+1}^0, W_{n+1}^0)= \vol(\GL_n(\roi_E))L(1,\sigma_u \times \br{\sigma_u}).
\end{align}
Now by Equation (3) in Section (1.4) of \cite{Jacquet-Shalika_2} and Equation (5) of \cite{Matringe}, we have 
\begin{align*}
\lambda(W_n^0\otimes W_{n+1}^0)=\vol(\GL_n(\roi_E))L(\frac{1}{2}, \sigma_n \times \sigma_u) =\vol(\GL_n(\roi_E))L(\frac{1}{2}, \sigma_n \times \sigma_{n+1}) .
\end{align*}
Then it follows easily from definitions and \Cref{essential_whit_fcn} that 
\begin{align}
\label{calc_I}
I_\sigma(\id_{K^{\prime, c}})&=\lambda(\sigma(\id_{K^{\prime, c}})(W_n\otimes W_{n+1})) \br{\beta_n(W_n)\beta_{n+1}(W_{n+1})}\nonumber \\
&=\vol(K^{\prime, c})\lambda(W_n\otimes W_{n+1}) \br{\beta_n(W_n)\beta_{n+1}(W_{n+1})}\\
&=\vol(K^{\prime, c})\vol(\GL_n(\roi_E))c_nc_{n+1}L(\frac{1}{2},\sigma_{n} \times \sigma_{n+1}) \br{\beta_n(W_n)\beta_{n+1}(W_{n+1})}. \nonumber
\end{align}

On the other hand, it is known that for the spherical function $W_n$ we have \cite{Zhang14} 
\begin{equation}
\label{calc_unram_beta}
\beta_n(W_n)=c_n\vol(\GL_{n-1}(\roi_F))L(1,\sigma_n, As^{(-1)^{n-1}}).
\end{equation}
So it only remains to calculate $\beta_{n+1}(W_{n+1})$. 
This calculation is known to Matringe and Anandavardhanan \cite[Theorem 6.1]{Anandavardhanan-Matringe}, but we will repeat it here for reader's convenience.
Let $\sigma_u$, $r$, and $\chi_1,\cdots, \chi_r$ be as in \Cref{pi_unram} with respect to $\sigma_{n+1}$ and let $\alpha_i=\chi_i(\vp)$ for $1\leq i \leq r$. 
For $m\geq 1$, let $\delta_{E,m}$ and $\delta_{F,m}$ be the modular characters of the standard Borel subgroup of $\GL_m(E)$ and $\GL_m(F)$ respectively. 

\begin{lem}
\label{lem_beta}
Let $\sigma_{n+1}$ be a tempered generic representation of $\GL_{n+1}(E)$ and let $\psi_E$ be as defined in the beginning of \Cref{section_integral}. Let $W_{n+1}^0\in \W(\sigma_{n+1},\psi_E)$ be the normalized essential Whittaker function and let $\sigma_u$ be as in \Cref{pi_unram}. Then 
\[\beta_{n+1}(W_{n+1}^0)=\vol(\GL_n(\roi_F)) L(1, \sigma_u, As^{{(-1)}^n}).\]
\end{lem}

\begin{proof}
By \Cref{thm_ess_Whit} and \Cref{sph_whit},
\begin{align*}
 &\beta_{n+1}(W_{n+1}^0) \\
=&\int_{N_{n}(F)\setminus \GL_{n}(F)} W_{n+1}^0 (\epsilon_{n+1}(\tau) \mat{g&\\&1}) \eta_{E/F}(\det g)^{n} dg\\
=& \vol(\GL_n(\roi_F))\sum_{f=(f_1,\cdots, f_{n})\in \Z^{n}} 
W_{n+1}^0(\epsilon_{n+1}(\tau) \mat{\vp^f&\\&1})\delta_{F,n}^{-1}(\vp^f)\eta_{E/F}(\vp^{f_1+\cdots +f_{n}})^{n} \\
=&\vol(\GL_n(\roi_F)) \sum_{f_1\geq \cdots \geq f_r\geq 0} 
W_{\pi_u}^0(\vp^f) 
|\vp^{f_1+\cdots+f_r}|_E^{\frac{n+1-r}{2}} \delta_{F,n}^{-1}(\mat{\vp^f&\\&I_{n-r}})
(-1)^{n(f_1+\cdots +f_r)} \\
=& \vol(\GL_n(\roi_F)) \sum_{f_1\geq \cdots \geq f_r\geq 0} 
\delta_{E,r}(\vp^f)^{1/2}s_f(\alpha) |\vp^{f_1+\cdots+f_r}|_F^{n+1-r}
\delta_{F,n}^{-1}(\mat{\vp^{f}&\\ &I_{n-r}}) (-1)^{n(f_1+\cdots +f_r)}\\
=&\vol(\GL_n(\roi_F)) \sum_{f_1\geq \cdots \geq f_r\geq 0} 
\left(\prod_{i=1}^r |\vp^{f_i}|_E^{(r+1-2i)/2} s_f(\alpha)\right) 
|\vp|_F^{(f_1+\cdots+f_r)(n+1-r)}\\
&\hspace{8cm} \left(\prod_{i=1}^r |\vp^{f_i}|_F^{-(n+1-2i)} \right) 
(-1)^{n(f_1+\cdots +f_r)} \\
=&\vol(\GL_n(\roi_F))\sum_{f_1\geq \cdots \geq f_r\geq 0} s_f(\alpha) (-1)^{n(f_1+\cdots +f_r)}|\vp|_F^{f_1+\cdots+f_r}\\
=&\vol(\GL_n(\roi_F))\sum_{f_1\geq \cdots \geq f_r\geq 0} s_f(\alpha) ((-1)^{n}|\vp|_F)^{f_1+\cdots + f_r}\\
=&\vol(\GL_n(\roi_F))\sum_{f_1\geq \cdots \geq f_r\geq 0} s_f((-1)^n q_F^{-1} \alpha) \\
=&\vol(\GL_n(\roi_F))\prod_{1\leq i\leq r}(1- (-1)^n q_F^{-1} \alpha_i)^{-1}\prod_{1\leq i<j\leq r} (1-q_F^{-2}\alpha_i \alpha_j)^{-1}\\
=&\vol(\GL_n(\roi_F)) L(1, \sigma_u, As^{{(-1)}^n}),
\end{align*}
where the second to last step follows from Example 4 in Section I.5 of \cite{macdonald}. 
\end{proof}

In general $L(1, \sigma_u, As^{{(-1)}^n})$ is not equal to $L(1,\sigma_{n+1}, As^{{(-1)}^n})$, so we have to keep the representation $\sigma_u$ in the final result. 
Combining \Cref{lem_beta} with Equations (\ref{eq_constants}), (\ref{calc_I}), and (\ref{calc_unram_beta}), we obtain the following lemma.
\begin{lem}
\label{lemma_calc_I}
Let $\sigma_n$ and $\sigma_{n+1}$ be irreducible tempered unitary representations of $\GL_n(E)$ and $\GL_{n+1}(E)$ respectively such that $\sigma_n$ is unramified and $\sigma_{n+1}$ has conductor $c$. 
Let $\sigma=\sigma_n \boxtimes \sigma_{n+1}$ and let $\sigma_u$ be the unramified representation derived from $\sigma_{n+1}$ as in \Cref{pi_unram}, then
\begin{align*}
I_\sigma(\id_{K^{\prime, c}})
&=\frac{\vol(K_n')\vol(K^{\prime, c}_{n+1})\vol(\GL_{n-1}(\roi_F))\vol(\GL_{n}(\roi_F))}{\vol(\GL_{n-1}(\roi_E))} \times\\
& \qquad \qquad \frac{L( \frac{1}{2}, \sigma_{n} \times \sigma_{n+1})
\br{L(1,\sigma_n, As^{(-1)^{n-1}})}\br{L(1,\sigma_u, As^{{(-1)}^n})}}{L(1,\sigma_n \times \br{\sigma_n})L(1,\sigma_u \times \br{\sigma_u})}.
\end{align*}
\end{lem} 

Now let notations be as in \Cref{section_intro} and we assume in addition $2|c-\epsilon$.
Recall that with the basis $\mathcal{B}_W$ and $\mathcal{B}_V$ chose in \Cref{notation_n_measure}, the hermitian matrices on $W$ and $V$ are $I_n$ and $\mat{I_n&\\&\vp^c}$ respectively.
Then $K_n=U(W)(\roi_F)$ and $K_{n+1}^c=U(V)(\roi_F)\cap \bm{K}_{n+1}^c$, so \Cref{thm:BP_I-J} and \Cref{hom_fl} imply that 
\[J_\pi(\id_K)=L(1,\eta_{E/F}) c_1 I_\sigma(\id_{K'}),\]
with $c_1$ as in \Cref{hom_fl}. Combining \Cref{lemma_calc_I} and Equations (\ref{vol_GL_OF}), (\ref{vol_GL_OE}), and (\ref{vol_kprime}), we arrive at the following lemma.

\begin{lem}
\label{lemma_calc_J}
With notations in \Cref{section_intro}, when $\epsilon$ and $c$ are of the same parity, we have
\[J_\pi(\id_{K^c})=C\frac{L(\frac{1}{2},\sigma_n \times \sigma_{n+1})\br{L(1,\sigma_n, As^{(-1)^{n-1}})L(1,\sigma_u, As^{(-1)^{n}})}}{L(1,\sigma_n \times \br{\sigma_n})L(1,\sigma_u \times \br{\sigma_u})},\]
where
\[C=\vol(K_n)^2 L(1,\eta_{E/F})q_F^{-c(n+1)}(1+q_F^{-n}).\]
\end{lem}

One thing to keep in mind is that for $c_1>c_2$, $K_{n+1}^{c_1}$ is not contained in $K_{n+1}^{c_2}$ even though $\bm{K}_{n+1}^{c_1} \subset \bm{K}_{n+1}^{c_2}$, this is because the matrix form of $K_{n+1}^c$ depends on the basis $\mathcal{B}_V$, which is chosen differently for different $c$'s.

\begin{lem}
\label{lemma_Satake}
Let $V$ be a hermitian space over $E$ and $\pi$ be a tempered representation of $U(V)$. Then the Satake parameters $\{a_1,\cdots, a_r\}$ of $BC(\pi)_u$ satisfy $\{a_1,\cdots, a_r\}=\{a_1^{-1}, \cdots, a_r^{-1}\}$.
\end{lem}

\begin{proof}
Let $M$ the $L$-parameter of $BC(\pi)$. As explained in \cite[Section 4]{GGP}, it can be decomposed into
\[M=\bigoplus V_i \otimes M_i + \bigoplus W_i \otimes N_i + \bigoplus U_i \otimes(P_i + (P_i^\sigma)^\vee),\]
where $M_i, N_i, P_i$ are irreducible representations of $WD(E)$, each $M_i$ is conjugate-dual of the same sign as $M$, each $N_i$ is conjugate-dual of opposite sign of $M$, and $V_i, W_i, U_i$ are multiplicity spaces. 
Each $M_i, N_i, P_i$ are of the form $\rho \otimes Sym^r$, where $\rho$ are representations of $W(E)$ (equivalently, $L$-parameters of supercuspidal representations of $\GL_n(E)$) and $Sym^r$ is the degree $r$ symmetric polynomial representation of $SL_2(\C)$ (see \cite[Chapter 12]{Getz-Hahn}).
So the $L$-parameter of $BC(\pi)_u$ is 
\[M_u=\bigoplus{'}\ V_i \otimes M'_i\ +\ \bigoplus{'}\ W_i \otimes N'_i\ +\ \bigoplus{'}\ U_i \otimes(P'_i + ({P'_i}^\sigma)^\vee),\]
where $\bigoplus' V_i \otimes M'_i$ means that the sum only runs over those $i$ such that $M_i=M'_i \otimes Sym^r$ and $M'_i$ is a 1-dimensional unramified representation of $W(E)$, and similarly for the other two terms. The Satake parameters of $BC(\pi)_u$ are images of the Frobenious element $\Fr_E \in \Gal(E)$ in the 1-dimensional unramified $M'_i, N'_i, P'_i$'s.

Let $s=\Fr_F \in W(F) \setminus W(E)$ and suppose that $M_i=M'_i \otimes Sym^r$, where $M'_i$ is a 1-dimensional unramified representation of $W(E)$. Since $M_i$ is conjugate-dual, by definition there is a non-degenerate bilinear form $B$ on $M_i$ such that 
\[B(\tau m, s\tau s^{-1}n) = B(m,n)\] 
for all $\tau \in WD(E)$ and $m,n \in M_i$. Take $\tau=\Fr_E$, then $\tau$ acts as multiplication by some $t \in \C^\times$ on $M_i'$ and acts trivially on $Sym^r$. Since $s \tau s^{-1} = \tau$, we have 
\[t^2B(a,b)=B(ta, tb)=B(\tau a, \tau b)=B(a,b), \quad a,b \in M_i,\]
i.e., $t=\pm 1$. Similarly for $N_i$.

Now if $P_i=P_i' \otimes Sym^r$ and $P_i'$ is an unramified 1-dimensional representation, then $(P_i^\sigma)^\vee=({P_i'}^\sigma)^\vee \otimes (Sym^r)^\vee = ({P_i'}^\sigma)^\vee \otimes Sym^r$.
Since $({P_i'}^\sigma)^\vee(\Fr_E)=(P_i')(\Fr_E)^{-1}$, the Satake parameters come in pairs.
\end{proof}

Let $\alpha_1, \cdots, \alpha_n$ be the Satake parameters of $\sigma_n$. Since $\sigma_n$ comes from base change, it is conjugate-dual of sign $(-1)^{n+1}$ (see \Cref{base_change}). Consequently, its Satake parameters satisfy $\{\alpha_1, \cdots, \alpha_n\}= \{\alpha_1^{-1}, \cdots, \alpha_n^{-1}\}$. Since $\sigma_n$ is a unitary representation, we have $\{\alpha_1^{-1}, \cdots, \alpha_n^{-1}\} = \{\br{\alpha}_1,\cdots, \br{\alpha}_n\}$, and it is straightforward to verify that 
\[\frac{\br{L(1,\sigma_n, As^{(-1)^{n-1}})}}{L(1,\sigma_n \times \br{\sigma_n})}=L(1,\sigma_n, As^{(-1)^{n}})^{-1}.\]
Since $\sigma_{n+1}$ is tempered, its $L$-parameter is bounded, and it follows from the proof of \Cref{lemma_Satake} that the $L$-parameter of $\sigma_u$ is also bounded. Hence $\sigma_u$ is tempered, thus unitary. By \Cref{lemma_Satake}, the Satake parameters $\beta_1,\cdots, \beta_r$ of $\sigma_u$ also satisfy $\{\beta_1,\cdots, \beta_r\} =\{\beta_1^{-1},\cdots, \beta^{-1}_r\}= \{\br{\beta}_1,\cdots, \br{\beta}_r\}$, and we can make the same cancellation as above for $\sigma_u$. This finishes the proof of \Cref{thm:mainthm}.

\section{Newforms}
\label{section_applications}

\subsection{Proof of \Cref{exist_newform}}
Notations in this section follows from those in \Cref{base_change} if unspecified. First we recall some notations and results from \cite{GGP} and \cite{BP_GGP}. 
In the rest of this section we fix a non-trivial additive character $\psi: E/F \to \C^\times$. Let $M,N$ be conjugate-dual representations of $WD(E)$ and let the local root number $\epsilon(M,\psi)$ be as in \cite[Section 5]{GGP}.
Let $B$ be a bilinear form on $M$ that makes it conjugate-dual, we define $C_M$ to be the centralizer in $\text{Aut}(M,B)$ of the image of $WD(E)$ in $\GL(M)$. If $M$ is an $L$-parameter of a representation of a unitary group, then this definition of $C_M$ is compatible with the one given in \Cref{base_change} (see \cite[Theorem 8.1]{GGP}).

For any semisimple $\alpha \in C_M$, let 
$M^\alpha = \{m \in M| \alpha m=-m\}$. Then $M^\alpha$ is a conjugate-dual representation of $WD(E)$ of the same sign as $M$.
We can define a character $\chi_N$ on $A_M$ by 
\[\chi_N(a) = \epsilon(M^\alpha \otimes N,\psi),\]
where $\alpha \in C_M$ is a lift of $a \in A_M$ that is semisimple.
Similarly define $\chi_M: A_N \to \C^\times$.

Let $W,V$ be hermitian forms over $E$ such that $W\subset V$ and $W^\perp$ is odd dimensional. Let $G=U(V) \times U(W)$.
A relevant inner form of $G$ is a group $G'=U(V') \times U(W')$ such that $\dim V=\dim V'$, $\dim W= \dim W'$, and ${W'}^\perp \subset V'$ and $W^\perp \subset V$ are isomorphic. There are two relevant inner forms of $G$ (including itself).

\begin{thm}[{\cite[Conjecture 17.3 and Theorem 19.1]{GGP}, \cite[Theorem 8.5.1]{BP_GGP}}]
\label{distinct_rep}
In every tempered Vogan $L$-packet of $G(F)$ there exists a unique representation $\pi$ of a relevant pure inner form $G'=U(V') \times U(W')$ such that $\hom_{H'}(\pi,\C)\neq 0$. Such $\pi$ is said to be distinguished. 
Moreover, let $M$ (resp. $N$) be a tempered $L$-parameter of $U(V)$ (resp. $U(W)$), then the unique distinguished $\pi$ in the packet $\Pi_{M}  \times \Pi_{N}$ of $G$ corresponds to the character $\chi_M \times \chi_N$ of $A_N \times A_M$ under the map $J_\psi$ in \Cref{LLC}.
\end{thm}

\begin{lem}
\label{unique_in_packet}
Let $V$ be a Hermitian space over $E$. At most one representation in each tempered Vogan $L$-packet of $U(V)$ contains newforms.
\end{lem}

\begin{proof}
When $V$ is odd-dimensional, Theorem 1.1 in \cite{Atobe-Oi-Yasuda} states that if an irreducible tempered representation of $U(V)$ contains newforms, then it is generic. The lemma now follows from the fact that every tempered Vogan $L$-packet of $U(V)$ contains a unique generic representation (see \cite[Corollary 9.2.4]{Mok}).
  
Below we assume $\dim V=n+1$ and $n$ is odd. We follow the approach in \cite{Atobe-Oi-Yasuda}.
Let $\Pi_\phi$ be a tempered Vogan $L$-packet of $U(V)$ with conductor $c$ and suppose that $\pi_{n+1} \in \Pi_\phi$ contains newforms. Without loss of genrality, we can assume $\pi_{n+1}$ is a representation of $U(V)$. Define $W, K_n, K^c_{n+1}$ as in \Cref{section_intro}.
Let $H=U(W)$ with diagonal embedding into $U(V) \times U(W)$.
  
We first show that there exists an unramified representation $\pi_{n}$ of $U(W)$ such that $\hom_H(\pi_{n+1}\otimes \pi_n,\C)\neq 0$. 
Choose $0\neq v \in \pi_{n+1}^{K^c_{n+1}}$ and let $f_{\pi_{n+1}}:U(V) \to \C$ be the matrix coefficient 
\[f_{\pi_{n+1}}: g \mapsto \pa{\pi_{n+1}(g)v,v}.\]
Since $f_{\pi_{n+1}}(1)\neq 0$, there exists $\phi\in C^\infty_c(U(W))$ such that 
\[\int_{U(W)}f_{\pi_{n+1}}(h)\br{\phi(h)}dh \neq 0.\]
By Plancherel theorem, following the argument in the proof of \cite[Lemma 12.5]{Gan_Savin_2012}, we see that there exists a representation $\pi_n$ of $U(W)$ and $w\in \pi_n$ such that the matrix coefficient $\phi_w: h \mapsto \pa{\pi_n(h)w,w}$ on $U(W)$ satisfies 
\[\int_{U(W)}f_{\pi_{n+1}}(h)\br{\phi_w(h)}dh \neq 0.\]
The integral above does not change when we replace $w$ by $\int_{K_n}\pi_n(k)w\ dk$, so we can assume $w \in \pi_n^{K_n}$. Thus $\pi_n$ is unramified and $\hom_{H}(\pi_{n+1}\boxtimes \pi_n,\C)\neq 0$.

Let $M$ be the $L$-parameter of $\pi_{n+1}$ and $N$ be the $L$-parameter of $\pi_n$. 
\Cref{distinct_rep} states that $\pi_{n+1}$ corresponds to $\chi_N$ under local Langlands correspondence.
Note that $\pi_n$ is unramified implies that $N$ is an unramified parameter \cite[Section 7.1]{Mok}, and \cite[Proposition 21.1]{GGP} implies that there exists a representation $N_1$ of $WD(E)$ such that $N=N_1+{}^\sigma N_1^\vee + triv$.
For all $a\in A_M$, since $M^a$ is conjugate-dual, by \cite[Proposition 5.1(2)]{GGP} we see 
\[\chi_N(a)=\epsilon(M^a \otimes N, \psi)= \epsilon(M^a\otimes N_1 + {}^\sigma(M^a\otimes N_1)^\vee + M^a, \psi)=\epsilon(M^a,\psi)\]
is independent of the choice of $N$. Hence the $\pi_{n+1} \in \Pi_M$ that contains newforms is unique. 
\end{proof}

\begin{rmk}
In the proof above, the $M^a$ can be quite arbitrary, so $\chi_N$ is not necessarily the trivial character. In other words, when $V$ is even-dimensional, even if the $\pi_{n+1}$ containing newforms is a representation of the quasi-split unitary group, it is not necessarily the generic one with respect to $\psi$. This is different from the case of odd-dimensional hermitian spaces.
\end{rmk}

\begin{proof}[Proof of \Cref{exist_newform}]
Let $V$ be an $n+1$ dimensional hermitian space and $\phi_{n+1}$ be a tempered $L$-parameter of $U(V)$ and let $c$ be its conductor. Let $W$ be a subspace of $V$ such that $W=V\oplus^\perp \pa{e}$ with $\pa{e,e}=1$ if $c$ is even and $\pa{e,e}=\vp$ if $c$ is odd. Define $K^c$ with respect to $V,W$ as in \Cref{section_intro}. 

If disc $W=1$, let $W_0=W$ and  $V_0=V$. Let $W_1$ be the $n$ dimensional hermitian space with discriminant $-1$ and let $V_1=W_1\oplus^\perp \pa{e_1}$ with $\pa{e_1, e_1}=\pa{e,e}$.

If disc $W=-1$, let $W_1=W$ and $V_1=V$. Let $W_0$ be the $n$ dimensional hermitian space with discriminant 1 and define $V_0$ accordingly as above.

Let $G_i=U(W_i) \times U(V_i)$ and $H_i=U(W_i)$ with diagonal embedding into $G_i$ for $i=0,1$. Let $\phi_n$ be a tempered $L$-parameter of $U(W)$ with conductor 0. By \Cref{distinct_rep}, there is a unique $\pi=\pi_n \times \pi_{n+1} \in \Pi_{\phi_n \times \phi_{n+1}}$ that is distinguished.

If $\pi$ is a representation of $G_1$, then \Cref{hom_fl} and \Cref{lemma_calc_I} imply that 
\[0=J_\pi(0)=I_{BC(\pi)}(\id_{K^{\prime, c}})\neq 0,\]
which is impossible. So $\pi$ is a representation of $G_0$. Then \Cref{thm:mainthm} implies that
\[0\neq J_\pi(\id_{K^c})=\int_{H_0(F)} \tr(\pi(h)\pi(\id_{K^c}))dh .\]
In particular, there exists $h\in H_0(F)$ such that $\tr(\pi(h)\pi(\id_{K^c})) \neq 0$. Hence there exists $v \otimes w\in \pi_n \times \pi_{n+1}$ such that $\pi(\id_{K^c})(v\otimes w)\neq 0$. In other words,
\begin{align*}
  0\neq \pi(\id_{K^c})(v\otimes w) 
  &= \int_{K_n \times K^c_{n+1}} \pi(k_{n}\otimes k_{n+1}) (w \otimes v) d k_n d k_{n+1}\\
  &= (\int_{K_n} \pi_n(k_n)w dk_n) (\int_{K^c_{n+1}} \pi_{n+1}(k_{n+1})v dk_{n+1}).
\end{align*}
Then $\int_{K^c_{n+1}} \pi_{n+1}(k_{n+1})v dk_{n+1}$ is nonzero and contained in $\pi_{n+1}^{K^c_{n+1}}$, hence a newform.

Uniqueness of $\pi_{n+1}$ follows from \Cref{unique_in_packet} directly.
\end{proof}

\begin{rmk}
With $V,W,\Lambda_n$ in \Cref{section_intro}, let $\Lambda_{n+1}^m=\Lambda_n+\vp^m\roi_Ee$ and let $K_{n+1}^m$ be the kernel of the map $\Stab_{U(V)(F)}(\Lambda_{n+1}^m) \to \GL((\Lambda_{n+1}^m)^\vee/\Lambda_{n+1}^m)$, \cite{Atobe-Oi-Yasuda} also proved that $\pi_{n+1}^{K_{n+1}^m}=0$ for $m<\lfloor\frac{c}{2} \rfloor$. (The $K_{n+1}^{\lfloor\frac{c}{2} \rfloor}$ here is our $K_{n+1}^c$ in \Cref{section_intro}.)
Atobe \cite{atobe2023} proved results of a similar form (under his definition of newforms) when $V$ is even-dimensional and quasi-split. 
\end{rmk}

\subsection{Uniqueness of newforms for two-dimensional unitary groups}
\label{section_uniqueness}

When $V$ is odd-dimensional, Atobe, Oi, and Yasuda \cite{Atobe-Oi-Yasuda} proved that in each tempered Vogan $L$-packet of $U(V)$, the generic representation is the unique one containing newforms and its subspace of newforms is one-dimensional. When $V$ is even-dimensional and $\pi$ is a representation of $U(V)(F)$ containing newforms, the dimension of the space of newforms is unknown. 

However, when $\dim V=2$, we know the uniqueness (up to scalar) of newforms. 
Indeed, let $V^+$ (resp. $V^-$) be a quasi-split (resp. non quasi-split) 2-dimensional hermitian space over $E$ and let $\Pi_2$ be a tempered Vogan L-packet of $U(V^\pm)$.  Let $K_2^\pm$ be compact subgroups of $U(V^\pm)(F)$ in the definition of newforms in \Cref{section_intro}. Note that the 1-dimensional unitary group $U(1)$ is compact and we write $K_1=U(1)(F)$. Let $triv$ be the trivial representation of $U(1)(F)$. We have $K_1\subset K_2^\pm$ and 
\begin{align*}
&\sum_{\pi_2^+\in \Pi_2}\dim {\pi_2^+}^{K_2^+} + \sum_{\pi_2^-\in \Pi_2}\dim {\pi_2^-}^{K_2^-}\\
=& \sum_{\pi_2^+\in \Pi_2} \dim (\text{triv}^{K_1} \otimes {\pi_2^+}^{K_2^+}) +  \sum_{\pi_2^-\in \Pi_2} \dim (\text{triv}^{K_1} \otimes {\pi_2^-}^{K_2^-}) \\
\leq& \sum_{\pi_2\in \Pi_2} \dim (\text{triv} \otimes \pi_2)^{K_1} \leq 1,
\end{align*}
where $\pi_2^+$ (resp. $\pi_2^-$) runs through all representations of $U(V^+)(F)$ (resp. $U(V^-)(F)$) in $\Pi_2$ and the last inequality follows from \cite{Waldspurger1991}, \cite{Tunnell}, \cite{Saito1993}, and \cite{Prasad}. Hence the uniqueness of newforms in the special case of $\dim V=2$.

\subsection{Proof of \Cref{cor_I-I}}
Let $W,V,G,\pi$ be as in \Cref{section_intro}.
Note that 
\[J_\pi(f)=\sum_{v\in ON(\pi)} \alpha(\pi(f)v,v),\quad f\in C^\infty_c(G(F)),\]
where $ON(\pi)$ is an orthonormal basis of $\pi$.
When $\pi_{n+1}^{K^c_{n+1}}$ is one-dimensional, the space $\pi^{K^c}$ is also one-dimensional. For $\phi_0\in \pi^{K^c}$ with $\pa{\phi_0, \phi_0}=1$, choose an orthonormal basis $ON(\pi)$ of $\pi$ that contains $\phi_0$. Note that for any $\phi_0 \neq \phi \in ON(\pi)$, we have $\pi(\id_{K^c})\phi \in \pi^{K^c}$ and 
\[\pa{\pi(\id_{K^c})\phi,\phi_0}=\pa{\int_{K^c} \pi(k)\phi\ dk, \phi_0}
=\int_{K^c} \pa{\pi(k)\phi, \phi_0} dk
=\int_{K^c} \pa{\phi, \pi(k^{-1})\phi_0} dk
=0,\]
i.e., $\pi(\id_{K^c})\phi=0$.
Hence $\alpha(\phi_0, \phi_0)=J_\pi(\id_{K^c})$ and for arbitrary $\phi \in \pi^{K^c}$ we have 
\[\alpha(\phi, \phi)=\pa{\phi, \phi}J_\pi(\id_{K^c}).\]
Then \Cref{thm:mainthm} gives the explicit value of the local integral $\alpha(\phi, \phi)$ in this special case, and this concludes the proof of \Cref{cor_I-I}.
 
\bibliographystyle{alpha}
\bibliography{ref} 

\begin{thebibliography}{KMSW14}

\bibitem[AKY22]{atobe_kondo_yasuda_2022}
Hiraku Atobe, Satoshi Kondo, and Seidai Yasuda.
\newblock Local newforms for the general linear groups over a non-archimedean local field.
\newblock {\em Forum of Mathematics, Pi}, 10:e24, 2022.

\bibitem[AM17]{Anandavardhanan-Matringe}
U.~K. Anandavardhanan and Nadir Matringe.
\newblock Test vectors for local periods.
\newblock {\em Forum Mathematicum}, 29(6):1245--1260, 2017.

\bibitem[AOY22]{Atobe-Oi-Yasuda}
Hiraku Atobe, Masao Oi, and Seidai Yasuda.
\newblock Local newforms for generic representations of unramified odd unitary groups and fundamental lemma, 2022.

\bibitem[Ato23]{atobe2023}
Hiraku Atobe.
\newblock {Local newforms for generic representations of unramified even unitary groups I: Even conductor case}, 2023.

\bibitem[BG92]{B-G_spherical-whit-fcn}
Daniel Bump and David Ginzburg.
\newblock Symmetric square {{\(L\)}}-functions on {{\(GL(r)\)}}.
\newblock {\em Ann. Math. (2)}, 136(1):137--205, 1992.

\bibitem[BP15]{BP_GGP}
R.~Beuzart-Plessis.
\newblock {Endoscopie et conjecture locale raffinée de Gan–Gross–Prasad pour les groupes unitaires}.
\newblock {\em Compositio Mathematica}, 151(7):1309–1371, 2015.

\bibitem[BP21a]{BP_fl-proof}
Rapha{\"e}l Beuzart-Plessis.
\newblock {A new proof of the Jacquet–Rallis fundamental lemma}.
\newblock {\em Duke Mathematical Journal}, 170(12):2805 -- 2814, 2021.

\bibitem[BP21b]{beuzart-plessis_2021}
Raphaël Beuzart-Plessis.
\newblock Comparison of local relative characters and the {I}chino–{I}keda conjecture for unitary groups.
\newblock {\em Journal of the Institute of Mathematics of Jussieu}, 20(6):1803–1854, 2021.

\bibitem[DZ]{DZ}
Daniel Disegni and Wei Zhang.
\newblock {G}an–{G}ross–{P}rasad cycles and derivatives of $p$-adic {$L$}-functions preliminary version.
\newblock \url{http://disegni-daniel.perso.math.cnrs.fr/prtf-ggp.pdf}.
\newblock Accessed: 2023-10-30.

\bibitem[Fli88]{Flicker}
Yuval~Z. Flicker.
\newblock Twisted tensors and {E}uler products.
\newblock {\em Bulletin de la Soci{\'e}t{\'e} Math{\'e}matique de France}, 116:295--313, 1988.

\bibitem[GGP12]{GGP}
Wee~Teck Gan, Benedict~H. Gross, and Dipendra. Prasad.
\newblock Symplectic local root numbers, central critical $l$-values, and restriction problems in the representation theory of classical groups.
\newblock {\em Sur les Conjectures de Gross et Prasad. I, Astérisque}, (346):1--109, 2012.

\bibitem[GH24]{Getz-Hahn}
Jayce Getz and Hahn Heekyoung.
\newblock {\em {An Introduction to Automorphic Representations With a view toward trace formulae}}.
\newblock Springer, 2024.

\bibitem[GJ72]{G-J}
R.~Godement and H.~Jacquet.
\newblock {\em Zeta functions of simple algebras}.
\newblock Lecture Notes in Mathematics, Vol. 260. Springer-Verlag, Berlin-New York, 1972.

\bibitem[Gro18]{Gross}
Benedict~H Gross.
\newblock On generic representations of odd unitary groups.
\newblock unpublished, 2018.

\bibitem[GS12]{Gan_Savin_2012}
Wee~Teck Gan and Gordan Savin.
\newblock {Representations of metaplectic groups I: epsilon dichotomy and local Langlands correspondence}.
\newblock {\em Compositio Mathematica}, 148(6):1655–1694, 2012.

\bibitem[GZ86]{Gross-Zagier}
Benedict~H. Gross and Don~B. Zagier.
\newblock Heegner points and derivatives of {$L$}-series.
\newblock {\em Inventiones mathematicae}, 84(2):225--320, Jun 1986.

\bibitem[Har12]{Harris}
R.~Neal Harris.
\newblock {The Refined Gross–Prasad Conjecture for Unitary Groups}.
\newblock {\em International Mathematics Research Notices}, 2014(2):303--389, 10 2012.

\bibitem[JPSS81]{J-PS-S}
H.~Jacquet, I.~I. Piatetski-Shapiro, and J.~Shalika.
\newblock Conducteur des repr{\'e}sentations du groupe lin{\'e}aire.
\newblock {\em Mathematische Annalen}, 256(2):199--214, Jun 1981.

\bibitem[JR11]{J-R}
H.~Jacquet and S~Rallis.
\newblock {On the Gross-Prasad conjecture for unitary group}.
\newblock {\em Clay Mathematics Proceedings: On certain $L$-functions}, 13, 2011.

\bibitem[JS81a]{Jacquet-Shalika_2}
H.~Jacquet and J.~A. Shalika.
\newblock {On Euler Products and the Classification of Automorphic Forms II}.
\newblock {\em American Journal of Mathematics}, 103(4):777--815, 1981.

\bibitem[JS81b]{Jacquet-Shalika}
H.~Jacquet and J.~A. Shalika.
\newblock {On Euler Products and the Classification of Automorphic Representations I}.
\newblock {\em American Journal of Mathematics}, 103(3):499--558, 1981.

\bibitem[KMSW14]{KMSW}
Tasho Kaletha, Alberto Minguez, Sug~Woo Shin, and Paul-James White.
\newblock {Endoscopic Classification of Representations: Inner Forms of Unitary Groups}, 2014.

\bibitem[LR04]{Lansky-Raghuram}
Joshua Lansky and A.~Raghuram.
\newblock {Conductors and newforms for U(1,1)}.
\newblock {\em Proceedings Mathematical Sciences}, 114(4):319–343, November 2004.

\bibitem[Mac95]{macdonald}
I.G. Macdonald.
\newblock {\em {Symmetric Functions and Hall Polynomials}}.
\newblock Oxford mathematical monographs. Oxford University Press, 2 edition, 1995.

\bibitem[Mat09]{matringe_asai}
Nadir Matringe.
\newblock {Distinction and Asai $L$-functions for generic representations of general linear groups over $p$-adic fields}, 2009.

\bibitem[Mat13]{Matringe}
Nadir Matringe.
\newblock {Essential Whittaker functions for $\mathrm{GL}(n)$}.
\newblock {\em Doc. Math.}, 18:1191--1214, 2013.

\bibitem[Mok15]{Mok}
Chung~Pang Mok.
\newblock {Endoscopic Classification of representations of Quasi-Split Unitary Groups}.
\newblock {\em Memoirs of the American Mathematical Society}, 235(1108), 2015.

\bibitem[Pra94]{Prasad}
Dipendra Prasad.
\newblock {On an extension of a theorem of Tunnell}.
\newblock {\em Compositio Mathematica}, 94(1):19--28, 1994.

\bibitem[RS07a]{Rallis-Schiffmann}
Steve Rallis and Gérard Schiffmann.
\newblock Multiplicity one conjectures, 2007.

\bibitem[RS07b]{Roberts-Schmidt}
Brooks Roberts and Ralf Schmidt.
\newblock {\em Local newforms for GSp(4)}, volume 1918.
\newblock Springer, 01 2007.

\bibitem[RSZ17]{RSZ2}
M.~Rapoport, B.~Smithling, and W.~Zhang.
\newblock On the arithmetic transfer conjecture for exotic smooth formal moduli spaces.
\newblock {\em Duke Mathematical Journal}, 166(12), sep 2017.

\bibitem[RSZ18]{RSZ}
M.~Rapoport, B.~Smithling, and W.~Zhang.
\newblock Regular formal moduli spaces and arithmetic transfer conjectures.
\newblock {\em Mathematische Annalen}, 370(3):1079--1175, Apr 2018.

\bibitem[Sai93]{Saito1993}
Hiroshi Saito.
\newblock {On Tunnell’s formula for characters of $GL(2)$}.
\newblock {\em Compositio Mathematica}, 85(1):99--108, 1993.

\bibitem[Tsa16]{Tsai}
Pei-Yu Tsai.
\newblock Newforms for odd orthogonal groups.
\newblock {\em Journal of Number Theory}, 161:75--87, 2016.
\newblock Special Issue on Applications of Automorphic Forms in Number Theory and Combinatorics.

\bibitem[Tun83]{Tunnell}
Jerrold~B. Tunnell.
\newblock {Local $\epsilon$-Factors and Characters of GL(2)}.
\newblock {\em American Journal of Mathematics}, 105(6):1277--1307, 1983.

\bibitem[Wal85]{Waldspurger1985}
J.-L. Waldspurger.
\newblock {Sur les valeurs de certaines fonctions $L$ automorphes en leur centre de symétrie}.
\newblock {\em Compositio Mathematica}, 54(2):173--242, 1985.

\bibitem[Wal91]{Waldspurger1991}
Jean-Loup Waldspurger.
\newblock {Correspondance de Shimura et quaternions}.
\newblock {\em Forum mathematicum}, 3(3):219--308, 1991.

\bibitem[Yun11]{Yun}
Zhiwei Yun.
\newblock {The fundamental lemma of Jacquet and Rallis. With an appendix by Julia Gordon.}
\newblock {\em Duke Mathematical Journal}, 156(2):167 -- 227, 2011.

\bibitem[Zel80]{Zelevinsky}
A.~V. Zelevinsky.
\newblock {Induced representations of reductive ${\mathfrak {p}}$-adic groups. II. On irreducible representations of $\GL(n)$}.
\newblock {\em Annales scientifiques de l'École Normale Supérieure}, 13(2):165 -- 210, 1980.

\bibitem[Zha12]{Zhang_AFL}
Wei Zhang.
\newblock On arithmetic fundamental lemmas.
\newblock {\em Inventiones mathematicae}, 188(1):197--252, Apr 2012.

\bibitem[Zha14]{Zhang14}
Wei Zhang.
\newblock Automorphic period and the central value of {R}ankin-{S}elberg {$L$}-function.
\newblock {\em Journal of the American Mathematical Society}, 27(2):541--612, 2014.

\bibitem[Zha21]{Zhang_afl21}
Wei Zhang.
\newblock {Weil representation and Arithmetic Fundamental Lemma}.
\newblock {\em Annals of Mathematics}, 193(3):863 -- 978, 2021.

\bibitem[Zha22]{zzy}
Zhiyu Zhang.
\newblock Maximal parahoric arithmetic transfers, resolutions and modularity, 2022.

\end{thebibliography}
\printindex

\end{document}